\documentclass[11pt,reqno,a4paper]{article}
\usepackage{amsmath}
\usepackage{amsbsy}
\usepackage{amsthm}
\usepackage{amssymb}
\usepackage{amscd}
\usepackage{mathabx}
\usepackage{accents}
\usepackage{float}
\usepackage{url}
\usepackage{tikz}
\usetikzlibrary{matrix,arrows,decorations.pathmorphing}
 \usepackage{relsize}
\usepackage{xcolor, mathrsfs}

\usepackage{hyperref}
\hypersetup{
pdftitle={},%
pdfauthor={},%
pdfsubject={},%
pdfkeywords={},%
colorlinks=true,%
linkcolor={black},%
linktoc={},
linktocpage={},%
pageanchor={},
citecolor={black},
}

\usepackage[english]{babel}
\usepackage[utf8]{inputenc}
\usepackage[margin=2.41cm]{geometry}

\usepackage{hyperref}
\usepackage{url}

\usepackage{cite}

\usepackage{titlesec}
\titleformat{\section}{\large\bfseries\filcenter}{\thesection}{1em}{}
\titleformat{\subsection}{\bfseries}{\thesubsection}{1em}{}

\input{xy}
\xyoption{all}

\newtheorem{theorem}{Theorem}[section]
\newtheorem{corollary}[theorem]{Corollary}
\newtheorem{lemma}[theorem]{Lemma}
\newtheorem{proposition}[theorem]{Proposition}
\theoremstyle{remark}

\theoremstyle{definition}
\newtheorem{remark}[theorem]{Remark}
\newtheorem{definition}[theorem]{Definition}

\newtheorem{notation}[theorem]{Notation}
\newtheorem{setup}[theorem]{Setup}

\numberwithin{equation}{section}
\allowdisplaybreaks[1]

\catcode`@=12

\renewcommand\thanks[1]{%
  \begingroup
  \renewcommand\thefootnote{}\footnote{#1}%
  \addtocounter{footnote}{-1}%
  \endgroup
}

\renewcommand{\tilde}{\widetilde}
\renewcommand{\epsilon}{{\varepsilon}}

\usepackage{bbm}

\newcommand{\resp}{\textit{resp. }}
\newcommand{\ie}{\textit{i.e. }}
\newcommand{\cf}{\textit{cf. }}

\newcommand{\KK}{{\mathbb{K}}}
\newcommand{\RR}{{\mathbb{R}}}
\newcommand{\CC}{{\mathbb{C}}}
\newcommand{\n}{{\mathtt{n}}}

\def\bkappa{\kappa^\varrho}

\newcommand{\Dir}{{\mathsf{D}}}
\newcommand{\G}{\mathsf{G}}
\newcommand{\Id}{{\operatorname{Id}}}
\newcommand{\sol}{\mathsf{Sol}\,}

\newcommand{\tr}{\mathrm{tr}}

\newcommand{\E}{\mathsf{E}}
\newcommand{\I}{\mathsf{I}}
\newcommand{\M}{\mathsf{M}}
\newcommand{\N}{\mathsf{N}}
\renewcommand{\P}{\mathsf{P}}
\newcommand{\Q}{\mathsf{Q}}
\newcommand{\R}{\mathsf{R}}
\renewcommand{\S}{\mathsf{S}}
\newcommand{\T}{\mathsf{T}}
\newcommand{\U}{\mathsf{U}}
\newcommand{\V}{\mathsf{V}}
\newcommand{\Z}{\mathsf{Z}}

\newcommand{\oS}{\mathsf{S}}
\newcommand{\ff}{{\mathfrak{f}}}

\newcommand{\fh}{{\mathfrak{h}}}

\newcommand{\fR}{\mathfrak{R}}

\newcommand{\fA}{{\mathfrak{A}}}

\newcommand{\fI}{{\mathfrak{I}}}

\renewcommand{\aa}{{\mathfrak{a}}}

\newcommand{\vol}{{\textnormal{vol}\,}}
\newcommand{\supp}{{\textnormal{supp\,}}}
\newcommand{\Spin}{\textnormal{Spin}}

\newcommand{\bracket}[2]{\langle  #1\,|\, #2  \rangle}
\newcommand{\fiber}[2]{\prec  #1\,|\, #2  \succ}

\newcommand{\scalar}[2]{(#1\, |\, #2)}

\begin{document}

\begin{flushright}

\baselineskip=4pt

\end{flushright}

\begin{center}
\vspace{5mm}

{\Large\bf INTERTWINING OPERATORS FOR\\[3MM] SYMMETRIC HYPERBOLIC SYSTEMS ON\\[4MM] GLOBALLY HYPERBOLIC MANIFOLDS}

\thanks{
 S.M. and D.V. acknowledge the support of the INFN-TIFPA project ``Bell''.}

\vspace{5mm}

{\bf by}

\vspace{5mm}

{  \bf  Simone Murro and Daniele Volpe}\\[1mm]
\noindent  {\it Dipartimento di Matematica, Universit\`a di Trento and INFN-TIFPA}\\
{\it Via Sommarive 14,} {\it I-38123 Povo, Italy}\\[1mm]
email: \ {\tt  simone.murro@unitn.it, daniele.volpe@unitn.it}
\\[5mm]

Correspondence to be sent to: simone.murro@gmail.com
\\[5mm]
\end{center}

\begin{abstract}
In this paper, a geometric process to compare solutions of symmetric hyperbolic systems on (possibly different) globally hyperbolic manifolds is realized via a family of intertwining operators. By fixing a suitable parameter, it is shown that the resulting intertwining operator preserves Hermitian forms naturally defined on the space of homogeneous solutions. As an application, we investigate the action of the intertwining operators in the context of algebraic quantum field theory. In particular, we provide a new geometric proof for the existence of the so-called Hadamard states on globally hyperbolic manifolds. 
\end{abstract}

\\[0.5mm]

\renewcommand{\thefootnote}{\arabic{footnote}}
\setcounter{footnote}{0}

\section{Introduction}

Symmetric hyperbolic systems are an important class of first-order linear differential operators acting on sections of vector bundles on Lorentzian manifolds. The most prominent examples are the classical Dirac operator and the geometric wave operator, which can be understood by reducing a suitable second-order normally hyperbolic differential operator to a first-order differential operator. In the class of  Lorentzian manifolds with empty boundary known as \emph{globally hyperbolic}, the Cauchy problem of a symmetric hyperbolic system is well-posed. As a consequence the existence of advanced and retarded Green operators in guaranteed. These operators are of essential importance in the quantization of a classical field theory: Indeed they implement the canonical commutation relation for a bosonic field theory or the canonical anti-commutation relation for fermionic field theory. Moreover, their difference, dubbed causal propagator (or Pauli-Jordan commutator), can be used to construct quantum states. For further details, we recommend the recent reviews~\cite{CQF1,BD,FK}. 

In this paper, we investigate the existence of a geometrical map connecting the space of solutions of different symmetric hyperbolic systems over (possibily different) globally hyperbolic manifolds. 
A summary of the main result
obtained is the following (\cf Theorem~\ref{thm:main}):
\begin{theorem}\label{thm:mainintr}
Let $\alpha\in\{0,1\}$ and $\M_\alpha=(\M,g_\alpha)$ be globally hyperbolic manifolds admitting the same Cauchy temporal function. Consider the symmetric hyperbolic systems $\oS_\alpha$ over $\M_\alpha$ acting on sections of a real (or complex) vector bundle $\E_\alpha$ endowed with a non-degenerate sesquilinear fiber metric.  If there exists a vector bundle isometry $\kappa: \E_0\to\E_1$ and the set of timelike vectors for $g_1$ is contained in the one for $g_0$, then the spaces of inhomogeneous solutions for $\oS_0$ and $\oS_1$ are isomorphic. 
\end{theorem}
\begin{remark}
As we shall see in more details in Theorem~\ref{thm:main}, actually there exists a one-parameter family of isomorphisms $\R^\varrho:\sol(\oS_0)\to \sol(\oS_1)$, where $\varrho\in C^\infty(\M,\RR)$ is a strictly positive smooth function which plays a fundamental role in the conservation of Hermitian forms naturally defined on the space of spacially compact solutions (\cf Section~\ref{sec:cons herm str}). Unfortunately, none of these isomorphisms is canonical, in the sense that they depends on the choice of a smooth function $\chi\in C^\infty(\M,[0,1])$.
\end{remark}
Let us briefly comment on the geometric setting. 
First of all, the globally hyperbolic manifolds and the vector bundles can coincide, as for the case of scalar wave equations propagating on the same manifold which differs by a smooth potential. In this case, our analysis incorporates the results of Dappiaggi and Drago in ~\cite{Moller}. If the manifolds do not coincides but the vector bundles do, then a relation between the causal cones for the different metrics is need it in order to defined a suitable `intertwining' symmetric hyperbolic systems (\cf Lemma~\ref{lem:Schi}). Finally, if also the vector bundles do not coincides, as for the case of spinor bundles, then an isometry is need it in order to compare the different symmetric hyperbolic systems. Let us recall that the existence of a spinor bundle isometry in the Lorentzian setting is guaranteed by the result of B\"ar, Gauduchon and Moroianu~\cite{Baer}.  As a by-product of our analysis, we expect that the \emph{intertwining operator} defined in Proposition~\ref{prop:conserv scal prod} can be used to probe spectral properties of the Riemannian Dirac operator as follows:
Consider two Riemannian Dirac operators $\Dir_\alpha$ acting on sections of spinors bundles $\S\Sigma_\alpha$ over a compact Riemannian manifold $(\Sigma,h_\alpha)$. By defining $(\M_\alpha=\RR\times \Sigma_\alpha, g_\alpha= - dt^2 + h_\alpha)$ we immediately obtain a globally hyperbolic manifold. On $\M_\alpha$ consider the following symmetric hyperbolic system
$$ \oS_\alpha= \partial_t - \imath \Dir_\alpha : \E_\alpha \to \E_\alpha \,,$$
where $\E_\alpha$ is also considered as a Hermitian vector bundle on $\M_\alpha$ via the pull-back along the
projection $\pi : \M_\alpha \to \Sigma_\alpha$. Since $\Dir_\alpha$ is essentially self-adjoint on $L^2(\S\Sigma_\alpha)$, any vector in the kernel of $\oS_\alpha$ can be written employing the spectral calculus, namely
$$\Psi_\alpha= \exp(-\imath\, t\, \Dir_\alpha) \ff_\alpha\,, =  \Big( \int_{\sigma(\Dir_\alpha)}e^{-\imath\, \omega\, t}   \, dE_\omega \Big) \ff_\alpha$$
where $\sigma( \Dir_\alpha)$ and $dE_\omega$ are respectively the spectral measure and the spectrum of $\Dir_\alpha$ and $\ff_\alpha$ are initial data. 
It follows that intertwining operators should interplay between the spectral measure of $\Dir_0$ and $\Dir_1$ respectively. Very recently, Capoferri and Vassiliev gave an explicit formula for the Dirac evolution operator on any $3$-dimensional oriented close Riemannian manifold in~\cite{MatteDirac}. Combining their results with ours, we expect to get a better understanding of eigenvalues of the Dirac operators. \medskip

Our main result has a deep implication in free quantum field theory over generic globally hyperbolic spacetimes. Indeed,
by denoting with $\fA^{CCR}_\alpha$ (\resp $\fA^{CAR}_\alpha$) the algebra of real scalar fields (\resp Dirac fields) over a globally hyperbolic spacetimes $\M_\alpha$, the isomorphism $\R_\varrho$ defined in Proposition~\ref{prop:conserv scal prod} (\resp Proposition~\ref{prop:conserv sympl form}) can be lift to a $*$-isomorphism $\fA^{CCR}_0 \simeq \fA^{CCR}_1$ (\resp $\fA^{CAR}_0 \simeq \fA^{CAR}_1$) --  \cf Theorem~\ref{thm:alg iso} (\resp Theorem~\ref{thm:alg iso2}).
Remarkably, the pullback of a quasifree state along this $*$-isomorphism preserves the singular structure of the two-point distribution associated to the state (\cf Theorem~\ref{thm:main appl}). This result  is used to provide a new geometrical proof of the existence of the so-called Hadamard states (\cf Corollary~\ref{cor:existence Hadamard}).\medskip

This paper is structured as follows. In Section~\ref{sec:symm hyp syst} we review well know-facts about symmetric hyperbolic systems. In particular in the Subsection~\ref{sec:spin geom} and~\ref{sec:GWO} we introduce respectively the (classical) Dirac operator and the geometric wave operator as main examples of symmetric hyperbolic systems. In Section~\ref{sec:main} we prove the main result of this paper and we investigate the conservation of Hermitian forms. Finally, Section~\ref{sec:applications} is devoted to analyzing the consequence of such isomorphism in the context of algebraic quantum field theory.

\subsection*{Notation and convention}
\begin{itemize}
\item[-] The symbol $\KK$ denotes on of the elements of the set $\{\RR,\CC\}$;
\item[-] $\M_g:=(\M,g)$ is a globally hyperbolic manifold  (with empty boundary) -- \cf Definition~\ref{def:globally hyperbolic} -- and $g$ has the signature $(-,+\dots,+)$;
\item[-]  $\mathcal{GH}_\M$ denotes the space of globally hyperbolic metrics on a smooth manifold $\M$ such that, for any $g_0 ,g_1\in \mathcal{GH}_\M$, $\M_0=(\M,g_0)$ and $\M_1=(\M,g_1)$ have the same Cauchy temporal function;
\item[-] $\E$ is a $\KK$-vector bundle over $\M_g$ with $N$-dimensional fibers, denoted by $\E_p$ for $p\in\M$, and endowed with a  nondegenerate sesquilinear fiber metric $\fiber{\cdot}{\cdot}:\Gamma(\E_p)\times\Gamma(\E_p)\to \KK$;
\item[-]  $\Gamma_{c}(\E), \Gamma_{pc}(\E), \Gamma_{fc}(\E), \Gamma_{tc}(\E),  \Gamma_{sc}(\E)$ resp. $\Gamma(\E)$  denote the spaces of compactly supported, past compactly supported, future compactly supported, timelike compactly supported, spacelike compactly supported resp. smooth sections of $\E$;
\item[-] $\oS:\Gamma(\E)\to\Gamma(\E)$ is a symmetric hyperbolic system -- \cf Definition~\ref{def:symm syst} -- and $\oS^\dagger$, $\oS^*$ are respectively the formal adjoint and the formal dual operator of $\oS$ -- \cf Remark~\ref{rmk:dual operator};
\item[-] when $\M$ is spin, $\S\M_g$ denotes the spinor bundle -- \cf Definition~\ref{def:spinor bundle} -- and $\Dir:\Gamma(\S\M_g)\to\Gamma(\S\M_g)$ denotes the classical Dirac operator -- \cf Definition~\ref{def:Dirac operator};
\item[-] for a fixed, but arbitrary $\ff\in\Gamma(\E)$, $\sol(\oS)$ denotes the space of inhomogeneous solutions 
$$\sol(\oS)=\{ \Psi\in\Gamma(\E) \;| \;  \oS\Psi=\ff \text{ with } \ff\in\Gamma(\E) \} \,,$$ 
while $\sol_{sc}(\oS)$ denotes the space of homogeneous solutions with spacially compact support
$$\sol_{sc}(\oS)=\{ \Psi\in\Gamma_{sc}(\E) \;| \; \Psi \in \ker \oS \} \,.$$ 
\end{itemize} 

\subsection*{Aknowledgements}

We would like to thank Claudio Dappiaggi, Nicol\`o Drago and Valter Moretti for helpful discussions related  to the topic of this paper. We are grateful to the referees for useful comments on the manuscript.

\section{Symmetric hyperbolic systems}\label{sec:symm hyp syst}

To goal of this section is to present a self-contained overview of symmetric hyperbolic systems and their properties on Lorentzian manifolds. For a detailed introduction, we recommend the lecture notes of B\"ar~\cite{Ba-lect}.\medskip

On a generic Lorentzian manifold, the Cauchy problem for a differential operator is in general ill-posed: This can be a consequence of the presence of closed timelike curves or the presence of naked singularities. Therefore, it is convenient to restrict ourself to the class of  \emph{globally hyperbolic manifolds}.
\begin{definition}
\label{def:globally hyperbolic}
A \emph{globally hyperbolic manifold} is a $(n + 1)$-dimensional, oriented,
time-oriented, smooth Lorentzian manifold $(\M,g)$ such that
\begin{itemize}
\item[(i)] There are no closed causal curves;
\item[(ii)] For every point $p,q\in\M$, $J^+(p)\cap J^-(q)$ is compact;
\end{itemize}
where $J^+(\U)$ (\resp $J^-(\U)$) denotes the set of points of $\M$ that can be reached by future (\textit{resp}. past) directed causal curves starting from $\U\subset \M$.
\end{definition}
\begin{notation}
For the rest of this section, $\M_g:=(\M,g)$ will always denote a globally hyperbolic manifold and we adopt the convention that the metric $g$ has signature $(-,+, \dots , +)$.
\end{notation}

The class of globally hyperbolic manifolds contains many important spacetimes, e.g. Minkowski spacetime, Friedmann-Robertson-Walker models, the Schwarzschild blackhole and de Sitter space. 
In \cite{Geroch}, Geroch established the equivalence for a Lorentzian manifold being global hyperbolic and the existence of a \emph{Cauchy hypersurface} $\Sigma$ (i.e. an achronal subset which is crossed exactly once by any inextendible timelike curve), which implies that
$\M$ is homeomorphic to $\R \times \Sigma$ and all Cauchy hypersurfaces are homeomorphic. The proof was carried out by finding a {Cauchy time function}, namely a continuous function  $t: \M \to \R$ which increases
strictly on any future-directed causal curve such that each level $t^{-1}(t_0)$,
$t_0 \in \R$, is a Cauchy hypersurface. In~\cite{BeSa} Bernal and S\'anchez ``smoothened'' the result of Geroch by introducing the notion of \emph{Cauchy temporal function}.
\begin{definition}
We say that a smooth time function  $t : \M \to \R$ is a \emph{Cauchy temporal function}  if
its gradient $\nabla t$ is past-directed timelike and its level set is a smooth Cauchy hypersurface.
\end{definition} 

\begin{theorem}[\cite{BeSa}, Theorem 1.1 and Theorem 1.2]\label{thm: Sanchez}
Any globally hyperbolic manifold admits a Cauchy temporal function. In particular, it is isometric to the smooth product manifold $\RR \times \Sigma$ with metric 
	$$ g= - \beta^2 d t^2 + h_t $$
	where $t:\RR\times \Sigma \to \RR$ is Cauchy temporal function,	 $\beta : 	\RR \times \Sigma \to \RR$ is a smooth positive function and $h_t$ is a Riemannian metric on each level set of $t$.
\end{theorem}

Let now $\E$ be a $\mathbb{K}$-vector bundle over a globally hyperbolic manifold $\M_g$ with finite rank $N$ and  endowed with a (possibly indefinite) non-degenerate sesquilinear fiber metric 
\begin{equation}\label{eq:fibermetric}
\fiber{\cdot}{\cdot}: \E_p\times\E_p\to\KK\,.
\end{equation}

\begin{definition}[\cite{Ba}, Definition 5.1]\label{def:symm syst}
 A linear differential operator $\oS \colon \Gamma(\E) \to \Gamma(\E)$ of first order is called a \emph{symmetric hyperbolic system} over $\M_g$ if 
\begin{enumerate}
\item[(S)] The principal symbol $\sigma_\oS (\xi) \colon \E_p \to \E_p$ is Hermitian with respect to $\fiber{\cdot}{\cdot}$ for every $\xi\in \T^*_p\M$ and for every $p \in \M$;
\item[(H)] For every future-directed timelike covector $\tau \in \T_p^*\M$, the bilinear form $\fiber{\sigma_\oS (\tau) \cdot}{\cdot}$ is positive definite on $\E_p$.
\end{enumerate}
\end{definition}
\begin{remark}
Notice that Definition~\ref{def:symm syst} depends on the fiber metric $\fiber{\cdot}{\cdot}$ and on the Lorentzian metric $g$ which defined the set of future-directed timelike covectors.
\end{remark}

Let us recall that for a first-order linear differential operator $\oS \colon \Gamma(\E) \to \Gamma(\E)$ the principal symbol $\sigma_\oS\colon \T^*\M \to \text{End}(\E)$ can be characterized by 
\begin{equation}\label{eq:princ simb}
\oS( f u) = f \oS u + 	\sigma_\oS (d f )u
\end{equation}
 where $u \in \Gamma(\E)$ and $f \in C^\infty(\M)$.
If we choose local coordinates $(t, x^1, \ldots, x^n)$ on $\M$, with $x^i$ local coordinates on $\Sigma_t$, and a local trivialization of $\E$, any linear differential operator $\oS\colon \Gamma(\E) \to \Gamma(\E)$ of first order reads in a point $p\in \M$ as
\begin{equation}\label{eq: SHS in chart}
\oS:=  A_0(p) \partial_t + \sum_{j=1}^n A_j(p) \partial_{x^j} + B(p)
\end{equation}
where the coefficients $A_0, A_j,B$ are $N\times N$ matrices, with $N$ being the rank of $\E$, depending smoothly on $p\in\M$.
In these coordinates, Condition~(S) in Definition~\ref{def:symm syst} reduces to 
$$\fiber{A_0\,\cdot}{\cdot}=\fiber{\cdot}{A_0\,\cdot} \qquad \text{and} \qquad \fiber{A_j\,\cdot}{\cdot}=\fiber{\cdot}{A_j\,\cdot}$$
 for $j=1,\dots, n$.  Condition (H) can be stated as follows: For any future directed, timelike covector $\tau=dt + \sum_j \alpha_j dx^j$,  $$ \fiber{\sigma_\oS(\tau)\cdot}{\cdot}=\fiber{(A_0 + \sum_{j=1}^{n} \alpha_j A_j)\cdot}{\cdot}$$
 defines a scalar product on $\E_p$.  \medskip
 \begin{remark}\label{rmk:dual operator}
Let $\oS^\dagger$ be the \emph{formal adjoint} operator of $\oS$, \ie the unique linear differential operator $\oS^\dagger:\Gamma(\E)\to\Gamma(\E)$ defined by
$$\int_{\M}\fiber{\Psi}{\oS\Phi} \vol_{\M}=\int_{\M}\fiber{\oS^{\dagger}\Psi}{\Phi} \vol_{\M}$$
for every $\Phi,\Psi\in\Gamma(\E)$ with $\supp\Phi\cap\supp\Psi$ compact. Its principal symbol $\sigma_{\oS^\dagger}$ satisfies
$$\sigma_{\oS^\dagger}(\xi)=-\sigma_{\oS}(\xi)^t\,,$$
for every $\xi\in\T^*\M$, see e.g.~\cite[Lemma 1.1.26]{Ba-lect}. By using property (S) in Definition~\ref{def:symm syst} it follows that $-\oS^\dagger$ is a symmetric hyperbolic system. Similarly, let  $\oS^*$ be the \emph{formal dual} operator of $\oS$, \ie the unique linear differential operator acting on sections of the dual vector bundle $\oS^*:\Gamma(\E^*)\to\Gamma(\E^*)$ defined by
$$\int_{\M} \Phi^*(\oS\Psi) \, \vol_{\M}=\int_{\M} (\oS^{*}\Phi^*)\Psi \, \vol_{\M}$$
for every $\Phi^*\in \Gamma(\E^*)$, $\Psi\in\Gamma(\E)$ with $\supp\Phi^*\cap\supp\Psi$ compact. By introducing the vector bundle isometry
$$\Upsilon:\E \to \E^* \qquad \Phi \mapsto  \Phi^*:=\fiber{\Phi}{\cdot} \,,$$
computations show that $\oS^*=\Upsilon \oS^\dagger \Upsilon^{-1}$. For further details we refer to ~\cite[Section 1.1.3]{Ba-lect}.
\end{remark}

\begin{theorem}[\cite{Ba}, Theorem 5.6 and Proposition 5.7]\label{thm:main shs}
 The Cauchy problem for a symmetric hyperbolic system $\oS$ on a globally hyperbolic manifold is well-posed, \ie  for any $\ff \in \Gamma_c(\E)$ and $\fh \in \Gamma_c(\E|_{\Sigma_0})$ there exists a unique smooth solution $\Psi\in\Gamma_{sc}(\E)$ with spatially compact support to the initial value problem
\begin{equation}\label{CauchyK}
\begin{cases}{}
{\oS }\Psi=\ff   \\
\Psi|_{\Sigma_0} = \fh   \\
\end{cases} 
\end{equation}
which depends continuously on the data $(\ff,\fh)$.
\end{theorem}

As a byproduct of the well-posedness of the Cauchy problem, it follows the existence of Green operators.

\begin{proposition}[\cite{Ba}, Theorem 5.9 and Theorem 3.8]\label{prop:Green}
A symmetric hyperbolic system is Green hyperbolic, \ie  there exist linear  maps, dubbed \emph{advanced Green operator}
$\G^+\colon \Gamma_{pc}(\E)  \to  \Gamma_{pc}(\E)$  and \emph{retarded Green operator} $\G^-\colon  \Gamma_{fc}(\E)  \to  \Gamma(\E)_{fc}$,  satisfying
\begin{itemize}
\item[(i.a)] $\G^+ \circ \oS\ff  = \oS\circ \G^+ \ff=\ff$ for all $ \ff \in  \Gamma_{pc}(\E)$ ,
\item[(ii.a)]  $\supp(\G^+ \ff ) \subset J^+ (\supp \ff )$ for all $\ff \in  {\Gamma}_{pc} (\E)$;
\item[(i.b)]  $\G^- \circ \oS\ff  = \oS\circ \G^- \ff=\ff$ for all $ \ff \in  \Gamma_{fc}(\E)$,
\item[(ii.b)]  $\supp(\G^- \ff ) \subset J^- (\supp \ff )$ for all $\ff \in  {\Gamma}_{fc} (\E)$.
\end{itemize}
\end{proposition}


\begin{remark}\label{rmk:dual green}
On account of Remark~\ref{rmk:dual operator} it immediately follows that the formal dual operator $\oS^*$ is Green hyperbolic
\begin{align*}
&\G{^+}^*\colon  \Gamma_{pc}(\E^*) \cap   \Gamma_{sc}(\E^*)  \to  \Gamma_{pc}(\E^*) \cap   \Gamma_{sc}(\E^*)  \\
& {\G^-}^*\colon  \Gamma_{fc}(\E^*) \cap   \Gamma_{sc}(\E^*)  \to  \Gamma_{pc}(\E^*) \cap   \Gamma_{sc}(\E^*) 
\end{align*}
A straightforward computation shows that 
\begin{align*}
&\int_\M ({\G^+}^* \Phi_1^*) (\Psi_1) \vol_\M =\int_\M \Phi_1^*(\G^-\Psi_1) \vol_\M \qquad
\int_\M ({\G^-}^* \Phi_2^*) (\Psi_2) \vol_\M =\int_\M \Phi_2^*(\G^+\Psi_2) \vol_\M\,,
\end{align*}
for every $\Phi_1^*\in \text{dom}({\G^+}^*)$, $\Phi_2^*\in \text{dom}({\G^-}^*)$, $\Psi_1\in \text{dom}({\G^-})$ and $\Psi_2\in \text{dom}({\G^+})$. 
\end{remark}

\begin{definition}\label{def:causal prop}
Let $\oS$ be a Green hyperbolic operator and denote with  $\G^+$ and $\G^-$ respectively the advanced and retared Green operator. We call  \emph{causal propagator} $ \G: {\Gamma}_{fc}(\E)\to {\Gamma}(\E)$ the operator defined by $ \G:=\G^+-\G^-$.
\end{definition}

The causal propagator characterize the space of solutions to the homogeneous Cauchy problem, namely for any $\ff\in\Gamma_{tc}(\E)$, $\Psi:=\G\ff \in\ker\oS$. The properties of the causal propagator are summarized in the following proposition.
\begin{proposition}[\cite{CQF1}, Theorem 3.5]\label{prop:causal prop}
Let $\G$ be the causal propagator for a Green hyperbolic operator $\oS:\Gamma(\E)\to\Gamma(\E)$. Then the following linear maps forms an exact sequence
$$ \{0\} \to \Gamma_{tc}(\E)\xrightarrow{\oS} \Gamma_{tc}(\E) \xrightarrow{\G} \Gamma(\E)\xrightarrow{\oS} \Gamma(\E) \to \{0\}\,.
$$
\end{proposition}

\subsection{Geometric examples} 
In this section, we shall review two of the most important examples of symmetric hyperbolic systems: the classical Dirac operator and the geometric wave operator. More examples can be found in~\cite{GerochBook,IgorSHS}, while for further details on spin geometry on Lorentzian spin manifold, we refer to~\cite{Baer,SpinGeom,DHP}.

\subsubsection{The classical Dirac operator}\label{sec:spin geom}
Let $\M_g$ be a globally hyperbolic manifold and assume to have a spin structure \ie  a twofold covering map from the $\Spin_0(1,n)$-principal bundle $\P_{\rm 
Spin_0}$  to the bundle of positively-oriented tangent frames $\P_{\rm SO^+}$ of 
$\M$ such that the following diagram is commutative:
\begin{flalign*}
\xymatrix{
\P_{\Spin_0} \times \Spin_0(1,n) \ar[d]_-{} \ar[rr]^-{} && \P_\Spin \ar[d]_-{} 
\ar[drr]^-{}   \\
\P_{\rm SO^+} \times \textnormal{SO}(1,n)   \ar[rr]^-{} &&  \P_{\rm SO^+}  
\ar[rr]^-{} &&  \M_g\,.
}
\end{flalign*}
The existence of spin structures is related to the topology of $\M_g$.
A 
sufficient (but not necessary) condition for the existence of a spin structure 
is the parallelizability of the manifold.
Therefore, since any $3$-dimensional 
orientable manifold is parallelizable, it follows by Theorem~\ref{thm: Sanchez} 
that any 4-dimensional globally hyperbolic manifold admits a spin structure.
Given a fixed spin structure, one can use the spinor representation to 
construct 
the {spinor bundle}
\begin{definition}\label{def:spinor bundle}
Let $\M_g$ be a (globally hyperbolic) spin manifold. The \emph{(complex) spinor bundle} is the complex vector bundle
$$\S\M_g:=\Spin_0(1,n)\times_\rho \CC^N$$
where $\rho: \Spin_0(1,n) \to \textnormal{Aut}(\CC^N)$ is the complex 
$\Spin_0(1,n)$ representation and $N:= 2^{\lfloor \frac{n+1}{2}\rfloor}$. 
\end{definition}
The spinor bundle is enriched with the following structure:
\begin{itemize}
\item[-]a natural $\Spin_0(1, n)$-invariant indefinite fiber metrics
\begin{equation*}\label{eq: spin prod}
\fiber{\cdot}{\cdot}: \S_p\M_g \times \S_p\M_g \to \CC;
\end{equation*}
\item[-] a \textit{Clifford multiplication}, \ie a \textit{Clifford multiplication}, \ie  a fiber-preserving map 
$$\gamma\colon \T\M\to \text{End}(\S\M_g)$$ 
which satisfies
 for all $p \in \M_g$, $u, v \in \T_p\M$ and $\psi,\phi\in \S_p\M_g$
\begin{equation} \label{eq:gamma symm}
 \gamma(u)\gamma(v) + \gamma(v)\gamma(u) = -2g(u, v)\Id_{\S_p\M_g}\, \quad \text{and}\quad \fiber{\gamma(u)\psi}{\phi}=\fiber{\psi}{\gamma(u)\phi}\,.
\end{equation}
\end{itemize}
Using the spin product~\eqref{eq: spin prod}, we denote as \emph{adjunction map}, the complex anti-linear vector bundle isomorphism by
\begin{equation}\label{eq:adj map}
\Upsilon_p:\S_p\M_g\to \S^*_p\M_g  \qquad \psi \mapsto \fiber{\psi}{\cdot}\,,
\end{equation}
where  $\S^*_p\M_g$ is the so-called \emph{cospinor bundle}, \ie  the dual bundle of $\S_p\M_g$.
\begin{definition}\label{def:Dirac operator}
The \textit{(classical) Dirac operator} $\Dir$ is the operator defined as the composition of the metric connection $\nabla^{\S\M_g}$ on $\S\M_g$, obtained as a lift of the Levi-Civita connection on $\T\M$, and the Clifford multiplication:
$$\Dir=\gamma\circ\nabla^{\S\M_g} \colon \Gamma(\S\M_g) \to \Gamma(\S\M_g)\,.$$
\end{definition}

In local coordinates and with a trivialization of the spinor bundle $\S\M_g$, the Dirac operator reads as
\begin{align*}
	\Dir \psi = \sum_{\mu=0}^{n}  \epsilon_\mu \gamma(e_\mu) \nabla^{\S\M_g}_{e_\mu} \psi\, 
\end{align*}
where  $\{e_\mu\}$ is a local Lorentzian-orthonormal frame of $\T\M$ and $\epsilon_\mu=g(e_\mu,e_\mu)=\pm 1$.
 
\begin{remark}\label{rmk:metric dep}
Note that unlike differential forms, the definition of spinors (and cospinors) requires the choice of a spin structure and it depends on the metric of the underlying manifold. 
\end{remark}

\begin{proposition}\label{prop:Dir SHS}
The classical Dirac operator $\Dir$ on globally hyperbolic spin manifolds $\M_g$ is a symmetric hyperbolic system.
\end{proposition}
\begin{proof}
 The principal symbol $\sigma_\Dir$ of the Dirac operator reads as
\begin{equation}\label{eq:sigmaD}
\sigma_\Dir(\xi) \psi=\gamma(\xi^\sharp) \psi
\end{equation}
where $\xi\in \Gamma(\T^*\M)$, $\psi\in\Gamma(\S\M_g)$ and $\sharp:\T^*\M\to \T\M$ is the musical isomorphism implemented by the Lorentzian metric. Therefore, Property (S) of Definition~\ref{def:symm syst} is verified on account of~\eqref{eq:gamma symm}, while Property (H) follows by~\cite[Proposition 1.1]{dimock}, provided that the spin product~\eqref{eq: spin prod} was chosen with the appropriate sign.
\end{proof}

\begin{remark}
Noticed that in the literature it is often used the canonical positive-definite scalar product $\fiber{}{}_{\CC^N}$ despite the indefinite, non-degenerate spin product~\eqref{eq: spin prod}. As a consequence, $\Dir$ is no longer a symmetric hyperbolic system, but $\gamma(\frac{\partial}{\partial t}) \Dir$ does satisfies the required properties, see e.g.~\cite{Ginoux-Murro-20,DiracMIT}.
\end{remark}

\begin{remark}\label{rmk:soldual}
By Theorem~\ref{thm:main shs}, it follows that the space of solutions $\sol(\Dir)$ of the Dirac equation is not trivial. Moreover, since the Dirac operator is formally skew-adjoint, \ie $\Dir=-\Dir^\dagger$, by Remark~\ref{rmk:dual operator} it follows that for any solution $\psi$ of the homogeneous Dirac equation, the adjunction map $\Upsilon$ realizes an isomorphism of vector spaces, namely
\begin{equation}\label{eq:sol Dir e Dir*}
\ker\Dir \ni \psi \mapsto  \Upsilon \psi \in \ker \Dir^* \qquad \text{and}\qquad \ker \Dir^* \ni \phi^* \mapsto \Upsilon^{-1}\phi^* \in \sol(\Dir) \,. 
\end{equation}
\end{remark}

\subsubsection{The geometric wave operator}\label{sec:GWO}

Let $\V$ be an Hermitian vector bundle of finite rank and consider a normally hyperbolic operator  $\P:\Gamma(\V)\to\Gamma(\V)$ , \ie a $2^{nd}$-order linear differential operator with principal symbol $\sigma_\P$ defined by 
$$\sigma_\P(\xi)=-g(\xi,\xi)\cdot\mathrm{Id}_\V\,,$$
for every $\xi\in \T^*\M$.
Following \cite[Remark 3.7.11]{Ba-lect}, we shall reduce $\P$ to a symmetric hyperbolic system, but first we assume, without loss of generality, that $\M_g=(\M,g)$ is given by 
$$\M:=\RR\times\Sigma \qquad g=-\beta^2 dt^2 + h_t$$
see Theorem~\ref{thm: Sanchez}. 
By \cite[Lemma 1.5.5]{wave}, there exists a unique (metric) connection $\nabla^\V$ on $\V$ and a unique endomorphism field $c\in\Gamma(\text{End}(\V))$ such that 
\begin{equation}\label{eq:wave type}
\P=\tr_g({\nabla^\V}{\nabla^\V})+c = \frac{1}{{\beta^2}}\nabla_{\partial_t}^2+b_0\nabla_{\partial_t}+(\nabla^\Sigma)^*\nabla^{\Sigma}+\nabla_b^\Sigma + c,
\end{equation}
where  $\nabla^\Sigma$ is defined by $\nabla_X^\Sigma :=\nabla^\V_X $ for all $X\in \T\Sigma$, while $b_0\in C^\infty(\M)$ and $b\in\Gamma(\T\Sigma)$ are given by
$$b_0:=\frac{1}{2\beta^2}\left(\mathrm{tr}_{h_t}(\partial_t h_t)-\frac{\partial_t\beta^2}{{\beta^2}}\right)  \qquad \text {and }\qquad b:=-\frac{1}{2\beta^2}\mathrm{grad}_{h_t}({\beta^2}) \,. $$
Equation~\eqref{eq:wave type} allows us to rewrite the Cauchy problem for $\P:\Gamma(\V)\to\Gamma(\V)$
\begin{equation}\label{eq:Cauchy P}
 \begin{cases}{}
{\P} u= f   \\
u|_{\Sigma_0} = h   \\
\nabla_{\partial_t} u|_{\Sigma_0} = h'   \\
\end{cases} 
\end{equation} 
 as a Cauchy problem for $\oS:\Gamma(\E) \to \Gamma(\E)$,  
\begin{equation}\label{eq:Cauchy oS} \begin{cases}{}
{\oS }\Psi:= (A_0 \nabla^\V_{\partial_t} + A_\Sigma \nabla^{\Sigma} + B) \Psi =  \ff   \\
\Psi|_{\Sigma_0} = \fh   \\
\end{cases} 
\end{equation} 
 where $\E$ is the Hermitian vector bundle $\E:=\V\oplus(\T^*\Sigma\otimes \V)\oplus \V$, $B\in\Gamma(\text{End}(\E))$  and
$$
\Psi:=\begin{pmatrix}
\nabla^\V_{\partial_t}u\\\nabla^\Sigma u\\ u
\end{pmatrix}, \qquad \ff := \begin{pmatrix}
f\\ 0 \\ 0
\end{pmatrix} \qquad
A_0:=\begin{pmatrix}
\frac{1}{{\beta^2}}&0&0\\0&1&0\\0&0&1
\end{pmatrix} \qquad A_\Sigma=\begin{pmatrix}
0&-\mathrm{tr}_{h_t}&0\\-1&0&0\\0&0&0
\end{pmatrix}\,.
$$
$$ C:=\begin{pmatrix}
b_0&b\lrcorner&c\\0&\frac{1}{2}h_t^{-1}\partial_th_t\lrcorner&R_{\partial_t,\cdot}\\-1&0&0
\end{pmatrix} \,.
$$
The Cauchy problem~\eqref{eq:Cauchy oS} should be read as follows: $\nabla_{\partial_t}\nabla^\Sigma u$ is defined
by 
$$\left(\nabla_{\partial_t}\nabla^\Sigma u\right)_X:=\nabla_{\partial_t}\nabla_X^\Sigma u-\nabla_{(\nabla_{\partial_t}X)^\Sigma}^\Sigma u$$
 for all $X\in\Gamma(\pi_2^*\T\Sigma)$.
The term $\nabla^\Sigma \Psi$ is a section of $(\T^*\Sigma\otimes \V)\oplus(\T^*\Sigma\otimes \T^*\Sigma\otimes \V)\oplus (\T^*\Sigma\otimes \V)\to\M$, the trace coefficient contracting $\T^*\Sigma\otimes\T^*\Sigma$ of course. 
The coefficient $\frac{1}{2}h_t^{-1}\partial_th_t\lrcorner$ is more or less the Weingarten map put into the $\T\Sigma$ slot.
The curvature tensor $R$ is that of $\nabla$ and is by convention given for all $X,Y\in \T\M$ by $R_{X,Y}=[\nabla_X,\nabla_Y]-\nabla_{[X,Y]}$.
As in \cite[Remark 3.7.11]{Ba-lect}, Conditions (S) and (H) can be easily checked. Hence $\oS$ is a symmetric hyperbolic system.
\begin{remark}
Notice that, while any solution $u$ of the Cauchy problem~\eqref{eq:Cauchy P} gives a solution $\Psi$ to the Cauchy problem~\eqref{eq:Cauchy oS}, the contrary does not hold. Indeed, the space of initial data for $\Psi$ is ``too large'' and some a suitable restriction has to be imposed. For further details we refer to \cite[Remark 3.7.11]{Ba-lect}.
\end{remark}

\section{Intertwining operators}\label{sec:main}
This section aims to generalize the results of Dappiaggi and Drago in~\cite{Moller} by constructing a geometric map between the solutions space of symmetric hyperbolic systems defined on (possibly different) vector bundles over (possibly different) globally hyperbolic manifolds. Since the construction of a vector bundle can depend in general on the metric of the underlying Lorentzian manifold, as for the case of classical Dirac operator, it became necessary first to find a path connecting different metrics. Despite the space of Lorentzian metrics on a fixed smooth manifold is not path-connected, when we restrict our attention to globally hyperbolic manifolds, we get the following result.

\begin{lemma}\label{lem:glob hyp}
Let $\mathcal{GH}_\M$ be the space of globally hyperbolic metrics on a smooth manifold $\M$ such that, for any $g_0 ,g_1\in \mathcal{GH}_\M$, $\M_0=(\M,g_0)$ and $\M_1=(\M,g_1)$ have the same Cauchy temporal function. Then $\mathcal{GH}_\M$ is convex.
\end{lemma}
\begin{proof}
Let $g_0,g_1\in \mathcal{GH}_\M$. Since $\M_0$ and $\M_1$ admit the same Cauchy temporal function, the there exists a isometric splitting $\M=\RR\times\Sigma$ with metric $g_\alpha= - \beta_\alpha^2 d t^2 + {h_t}_{\alpha}$, for $\alpha\in\{0,1\}$.
This in particular, shows that the convex linear combination $ g_{\lambda}:=\lambda g_1 +\left(1-\lambda\right)g_0$ for any $\lambda\in [0,1]$, is a globally hyperbolic metric.
\end{proof}

Keeping in mind Lemma~\ref{lem:glob hyp} we introduce the following setup, which we shall use through this section:

\begin{setup}\label{setup}
For $\alpha \in\{0,1\}$, we have the following:
\begin{itemize}
\item[-]  $\mathcal{GH}_\M$ denotes the space of globally hyperbolic metrics on a smooth manifold $\M$ such that, for any $g_0 ,g_1\in \mathcal{GH}_\M$, $\M_0=(\M,g_0)$ and $\M_1=(\M,g_1)$ have the same Cauchy temporal function;
\item[-]   $\M_\alpha:=(\M,g_\alpha)$, where $g_\alpha \in\mathcal{GH}_\M$ and $g_1 \leq g_0$ (\ie the set of timelike vectors for $g_1$ is contained in the one for $g_0$);
\item[-] $\E_\alpha$ is a $\KK$-vector bundle over $\M_\alpha$  with finite rank and endowed with a nondegenerate sesquilinear fiber metric $\fiber{\cdot}{\cdot}_\alpha$;
\item[-] $\kappa_{1,0}: \E_0\to \E_{1}$ is a fiberwise linear isometry of vector bundles and, for any $\varrho\in C^\infty(\M,\RR)$ striclty positive, we set $\bkappa_{1,0}:=\varrho\; \kappa_{1,0}$ and  $\bkappa_{0,1}:=\varrho^{-1}\; \kappa_{0,1}$ ;
\item[-]  $\oS_{0,1}^\varrho:\Gamma(\E_1)\to\Gamma(\E_1)$ is the operator defined by $\oS^\varrho_{0,1}:=\bkappa_{1,0}\oS_0\bkappa_{0,1}$;
\item[-] $\sol(\oS_\alpha)$ denotes the space of solutions  for the symmetric hyperbolic system $\S_\alpha$ over $\M_\alpha$
$$\sol(\oS_\alpha):=\{\Psi_\alpha\in \Gamma(\E_\alpha) \,| \,  \oS_\alpha\Psi_\alpha = \ff_\alpha \text{ with } \ff_\alpha\in\Gamma(\E_\alpha) \;  \} \,.$$
\end{itemize}
\end{setup}

To construct an intertwining operator, we need a preliminary lemma.

\begin{lemma}\label{lem:Schi}
Assume the Setup~\ref{setup}. For any  $\chi\in C^\infty(\M,[0,1])$, the operator defined by 
\begin{equation}
\label{eq:def S intert}
\oS^\varrho_{\chi,1} := (1- \chi) \,\oS^\varrho_{0,1} + \chi\oS_{1}: \Gamma(\E_1) \to \Gamma(\E_1) 
\end{equation}
is a symmetric hyperbolic system over $\M_1$.
\end{lemma}
\begin{proof}
Using the characterization of the principal symbols given as in Equation~\eqref{eq:princ simb}, for every $ \Psi_1 \in\Gamma(\E_1)$ and $f\in C^{\infty}(\M)$ we thus obtain
\begin{equation*}
\begin{gathered}
\sigma_{\oS_{\chi,1}^\varrho}(df) \Psi_1=
f \left((1-\chi)\oS_{0,1}^\varrho+\chi\oS_{1}\right)\Psi_1 - \left((1-\chi)\oS_{0,1}^\varrho+\chi\oS_{1}\right)(f\Psi_1)\\
=(1-\chi)\sigma_{\oS_{0,1}^\varrho}(df)\Psi_1 +\chi\sigma_{\oS_{1}}(df)\Psi_1 \,.
\end{gathered}
\end{equation*}
We first show that 
$\oS_{0,1}^\varrho $ is a symmetric hyperbolic system.
By Equation~\eqref{eq:princ simb},
for every $ \Psi_\lambda \in\Gamma(\E_\lambda)$ and $f\in C^{\infty}(\M)$ we thus obtain
\begin{equation}\label{eq:transf princ symb}
\sigma_{\oS_{0,\lambda}^\varrho}(df) \Psi_1=
f \oS_{0,1}^\varrho \Psi_1 - \oS_{0,1}^\varrho(f\Psi_1) =  \bkappa_{1,0} f  \oS_0  \Psi_0-  \bkappa_{1,0}  \oS_0(f \Psi_0) = \bkappa_{1,0}  \sigma_{\oS_0} (df)\Psi_0\,,
\end{equation}
where $\Psi_0=\bkappa_{0,1} \Psi_1$ and $\bkappa_{0,1}\bkappa_{1,0}=\Id$. 
Since $\kappa_{1,0}$ is a fiberwise linear isometry by assumption, it follows
\begin{align}
\fiber{\sigma_{\oS^\varrho_{0,1}}(\xi)\Psi_1}{\Psi_1}_{1} &=\fiber{\bkappa_{1,0}\sigma_{\oS_0}(\xi)\Psi_0}{\bkappa_{1,0}\Psi_0}_1=\varrho^2 \fiber{\sigma_{\oS_0}(\xi)\Psi_0}{\Psi_0}_0 \,,
\end{align} 
where $\fiber{\cdot}{\cdot}_0$ is a fiberwise pairing on $\E_0$. 
Using that $\oS_0$ is a  symmetric hyperbolic system, it follows immediately that $\oS_{0,1}^\varrho$ satisfies property (S) in Definition~\ref{def:symm syst}. Furthermore, since $g_1\leq g_0$, any timelike covector $\xi$ for $g_1$ is also a timelike covector for $g_0$. Therefore $\oS_{0,1}^\varrho$ satisfies also property (H) as well.  Hence it is a symmetric hyperbolic system. To conclude our proof, it is enough to notice that a convex linear combination of Hermitian operators is a Hermitian operator, and a convex linear combination of positive operators is a positive operator. Hence also $\oS_{\chi,1}^\varrho$ is a symmetric hyperbolic system.
\end{proof}

Building on Lemma~\ref{lem:Schi}, we now prove the main result of this paper.
\begin{theorem}\label{thm:main}
Assume the Setup~\ref{setup}. 
Consider two Cauchy hypersurfaces $\Sigma^\pm\subset \M_1$ such that $\Sigma_+ \subset J^+(\Sigma_-)$ and let $\chi \in C^\infty(\M_1,[0,1])$ be non-decreasing along any future-oriented timelike curve such that 
$$\chi|_{J^+(\Sigma_+)}=1 \,,\qquad \text{and } \qquad \chi|_{J^-(\Sigma_-)}=0 \,.$$
Finally consider the operators $\R_\pm: \Gamma(\E_1)\to \Gamma(\E_1)$ defined by 
$$\R_-:=\Id - \G_{1}^- (1-\chi) (\oS_1 - \oS_{0,1}^\varrho) \quad \text{and}\quad  \R_+:=\Id - \G_{\chi}^+ \chi (\oS_1 -\oS_{0,1}^\varrho) $$
where $\G^+_\chi$ is the advanced Green operator for the operator $\oS_{\chi,1}^\varrho$ defined by Equation~\eqref{eq:def S intert} and $\G^-_1$ is the retared Green operator for $\oS_1$.
If $\ff_1=\bkappa_{1,0} \ff_0$, then the intertwining operator 
$$\R_\varrho:=\R_-\circ \R_+ \circ \bkappa_{1,0}: \sol(\oS_0)\to \sol(\oS_1)$$
implements a non-canonical isomorphism.
\end{theorem}

\begin{proof}
Since, for any $\ff_1 \in  \Gamma(\E_1)$ , $\chi \ff_1 \in {\Gamma}_{pc}(\E_1)$ and $(1-\chi)\ff_1 \in  {\Gamma}_{fc}(\E_1)$ then the intertwining operator $\R_\varrho$ is well defined. Indeed, for any $\Psi_0 \in \sol(\oS_0)$, it turns out that 
\begin{gather*}
\chi (\oS_1-\oS_{0,1}^\varrho)\circ \bkappa_{1,0} \;\Psi_0 \in  {\Gamma}_{pc}(\E_1)= \text{dom\,}\G_\chi^+\\
(1-\chi) (\oS_1-\oS_{0,1}^\varrho)\circ\R_+\circ \bkappa_{1,0}\Psi_0 \in  {\Gamma}_{fc}(\E_1) = \text{dom\,}\G_1^- \,.
\end{gather*} 
 The smoothness of $\R_\varrho \Psi_0$ is a by-product of the regularity properties of the advanced and retarded Green operators.
By straightforward computation, we thus obtain
\begin{align}\label{eq:1}
\oS_1 \circ \R_-  &= \oS_1 -  (1-\chi) ( \oS_1 -\oS_{0,1}^\varrho) = \oS_{\chi,1}^\varrho\\ \label{eq:2}
\oS_{\chi,1}^\varrho \circ \R_+ & = \oS_{\chi,1}^\varrho -  \chi ( \oS_1 -\oS_{0,1}^\varrho ) =  \oS_{0,1}^\varrho  \,,
\end{align}
where we used that $\oS_1\circ \G_1^-=\Id$ and $\oS_{\chi,1}^\varrho\circ\G^+_\chi=\Id$ (\textit{cf.} Proposition~\ref{prop:Green}). 
In particular, this implies that  \begin{equation}\label{eq: S1RS0}
\oS_1\circ \R_{\varrho}= \bkappa_{1,0}\oS_0\,.
\end{equation}
Therefore, for any $\Psi_0\in\sol(\oS_0)$, it holds
$$ (\oS_1 \circ \R_{\varrho}) \Psi_0= \bkappa_{1,0}\oS_0 \Psi_0=\bkappa_{1,0}\ff_0 = \ff_1\,.$$
Hence $\R_\varrho\Psi_0\in\sol(\oS_1)$.
\medskip

To conclude our proof, it suffices to prove that $\R_-$ and $\R_+$ are  invertible. Indeed, by defining
$$ \R_{\varrho}^{-1}:= \bkappa_{0,1} \circ \R_+^{-1} \circ \R_-^{-1}  $$
it follows $ \R_{\varrho}\circ \R_{\varrho}^{-1} = \Id: \sol(\oS_1)\to\sol(\oS_1)$  and $ \R_{\varrho}^{-1}\circ \R_{\varrho} = \Id: \sol(\oS_0)\to\sol(\oS_0)$.
To this end, we make the following ansatz:
\begin{equation*}
\begin{gathered}
\R_-^{-1}:=\Id + \G_{\chi}^- (1-\chi) (\oS_1 - \oS_{0,1}^\varrho): {\Gamma}(\E_1)\to {\Gamma}(\E_1) \\
\R_+^{-1}:=\Id + \G_{{0,1}}^+ \chi (\oS_1 -\oS_{0,1}^\varrho): {\Gamma}(\E_1)\to {\Gamma}(\E_1)\,,
\end{gathered}
\end{equation*}
where $\G_{{0,1}}^+$ is the advanced Green operator for $\oS_{0,1}^\varrho$.
We begin by showing that $\R_{-}^{-1}$ is a right inverse for $\R_-$ 
\begin{align*}
\R_- \circ \R_-^{-1} &=\left(\Id - \G_{1}^- (1-\chi) (\oS_1 - \oS_{0,1}^\varrho)\right)\circ \left(\Id + \G_{\chi}^- (1-\chi) (\oS_1 - \oS_{0,1}^\varrho)\right) =\\
&=\Id-\G_{1}^- (\oS_1-\oS_{\chi,1}^\varrho) + \G_{\chi}^-(\oS_1-\oS_{\chi,1}^\varrho) -\G_{1}^- (\oS_1-\oS_{\chi,1}^\varrho) \G_{\chi}^- (\oS_1-\oS_{\chi,1}^\varrho)=\Id
\end{align*}
where we used $(1-\chi) (\oS_1 - \oS_{0,1}^\varrho)=\oS_1-\oS_{\chi,1}^\varrho$ together with $\G_{1}^- (\oS_1-\oS_{\chi,1}^\varrho) \G_{\chi}^-=\G_1^+-\G_{\chi}^-   $\,.
 Similarly we get
\begin{align*}
\R_+ \circ \R_+^{-1} &=\left(\Id - \G_{\chi}^+ \chi (\oS_1 -\oS_{0,1}^\varrho)\right)\circ \left(\Id + \G_{{0,1}}^+ \chi (\oS_1 -\oS_{0,1}^\varrho)\right) =\\
&=\Id - \G_{\chi}^+ (\oS_{\chi,1}^\varrho - \oS_{0,1}^\varrho)+ \G_{{0,1}}^+ (\oS_{\chi,1}^\varrho - \oS_{0,1}^\varrho)-\G_{\chi}^+ (\oS_{\chi,1}^\varrho - \oS_{0,1}^\varrho)\G_{{0,1}}^+ (\oS_{\chi,1}^\varrho - \oS_{0,1}^\varrho)=\Id
\end{align*}
where we used $\chi (\oS_1 - \oS_{0,1}^\varrho)=\oS_{\chi,1}^\varrho - \oS_{0,1}^\varrho$ together with $\G_{\chi}^+ (\oS_{\chi,1}^\varrho - \oS_{0,1}^\varrho)\G_{{0,1}}^+=\G_{{0,1}}^+ - \G^+_\chi $
Since $ \R_{-}^{-1} \circ \R_{-}$ and $ \R_{+}^{-1} \circ \R_{+}$ have analogous computations, we can conclude.
\end{proof}

\begin{remark}
By uniqueness of solution, $\R_\varrho$ implements the following geometric map:
Let be $\Psi_0\in\sol(\oS_0)$ and denote with $\Psi_{1,0}=\bkappa_{1,0}\Psi_0$. Then consider the operator which maps the Cauchy data $\Psi_{1,0}|_{\Sigma_-}$ to the corresponding Cauchy data on $\Sigma_+$ by evolving it via the evolution operator of $\oS_{\chi,1}^\varrho$, which exists on account of Lemma~\ref{lem:Schi} and Theorem~\ref{thm:main shs}. Finally define $\R_\varrho\Psi_0$ as the solution for $ \oS_1 \R_\varrho\Psi_0 = \ff_1 $ and Cauchy data provided by those previously obtained on $\Sigma_+$.  If the inhomogeneity $\ff_1=0$, then the evolution operator is an unitary operator from $L^2(\S\M|_{\Sigma_\alpha})\to L^2(\S\M|_{\Sigma_\beta})$, where $\Sigma_\alpha,\Sigma_\beta\subset \M$ are Cauchy hypersurfaces.
\end{remark}

On account of Remark~\ref{rmk:dual operator}, the formally adjoint operator $\oS^\dagger:\Gamma(\E)\to\Gamma(\E)$ is a symmetric hyperbolic system, clearly up to a sign. Therefore the results of Theorem~\ref{thm:main} applies immediately to $-\oS^\dagger$ and we denote  its intertwining operator by $\R_\varrho^\dagger$. By denoting $\Upsilon:\E\to\E^*$ we immediately get the following result.
\begin{proposition}\label{prop:R*}
Assume the setup of Theorem~\ref{thm:main} and let ${\bkappa_{1,0}}^*:\Gamma(\E^*_0) \to \Gamma(\E^*_1)$ be the linear isometry defined by ${\bkappa_{1,0}}^*:= \Upsilon_1 \bkappa_{1,0} \Upsilon^{-1}_0$. The intertwining operator  $\R^*_\varrho:=\Upsilon_1 \R_\varrho^\dagger \Upsilon_0^{-1}$ implements a non-canonical isomorphism between $\sol(\oS_0^*)$ and $\sol(\oS_1^*)$.
\end{proposition}
\begin{proof}
As shown in Remark~\ref{rmk:dual operator} it holds $\oS^*=\Upsilon \oS^\dagger \Upsilon^{-1}$, which should be reads as 
$$\int_{\M} \fiber{\oS_\alpha^\dagger\Phi_\alpha}{\Psi_\alpha} \, \vol_{\M}=\int_{\M} (\oS_\alpha^{*} \Upsilon_\alpha\Phi_\alpha)\Psi_\alpha \, \vol_{\M} \,,$$
for every $\Phi_\alpha \in \sol(\oS_\alpha^\dagger)$ and $\Psi_\alpha \in\Gamma_c(\E_\alpha)$.
Therefore, Equation~\eqref{eq: S1RS0} rewrite as
$$ \oS_1^*\Upsilon_1\R^\dagger_{\varrho}= \Upsilon_1 \oS^\dagger_1 \R^\dagger_\varrho= \Upsilon_1 \bkappa_{1,0}  \oS_0^\dagger = \Upsilon_1 \bkappa_{1,0} \Upsilon_0^{-1} \Upsilon_0 \oS_0^\dagger =  {\bkappa_{1,0}}^*\,  \oS_0^* \Upsilon_0.$$
Since for any $\Psi_0\in\sol(\oS_0^\dagger)$, it holds $\Upsilon_0 \Psi_0 \in \sol(\oS_0^*)$, we can conclude.
\end{proof}

\subsection{Conservation of Hermitian structures}\label{sec:cons herm str}

In this section, we are going to show that the intertwining operator $\R_\varrho$ preserves Hermitian structures on the spaces of \emph{homogeneous solutions with spacially compact support}, once that $\varrho\in C^\infty(\M)$ is chosen suitably.
Despite our result can be formulated abstractly for a generic symmetric hyperbolic system, we believe that analyzing the conservation of symplectic forms (for waves-like fields) and Hermitian scalar products (for Dirac fields) separately is more preparatory for Section~\ref{sec:applications}.

\subsubsection{Hermitian scalar products for Dirac fields}\label{sec:Dir-Sol-Herm}

As already underlined in~Remark~\ref{rmk:metric dep}, the space of spinors depends on the metric of the underlying manifold $\M_\alpha$. Therefore, an identification between spaces of sections of spinor bundles for different metrics is needed to construct an intertwining operator. This can be achieved by following~\cite[Section 5]{Baer}. \medskip

Let be $\lambda\in\RR$, $g_\lambda\in\mathcal{GH}_\M$ and consider a family of globally hyperbolic manifolds with the same Cauchy temporal function $\M_\lambda:=(\M,g_\lambda)$. Let $\Z$ be the Lorentzian manifold 
$$ \Z= \I \times \M \qquad\qquad g_\Z =  d\lambda^2 + g_\lambda\,.$$
 On $\Z$ there exists a globally defined vector field which we denote as $e_\lambda:=\frac{\partial}{\partial \lambda}$.
For any $\lambda$, the spin structures on $\Z$ and $\M_\lambda$ are in one-to-one correspondence: Any spin structure on $\Z$ can be restricted to a spin structure on $\M_\lambda$ and a spin structure on $\M_\lambda$ it can be pulled back on $\Z$ -- see~\cite[Section 3 and 5]{Baer}. 
Actually, the spinor bundle $\S\M_\lambda$ on each globally hyperbolic spin manifold $\M_\lambda$ can be identified with the restriction of the spinor bundle $\S\Z$ on $\M_\lambda$, in particular $\S\M_\lambda\simeq \S\Z|_{\M_\lambda}$ if $n$ is even, while $\S\M_\lambda\simeq \S^+\Z|_{\M_\lambda}\simeq \S^-\Z|_{\M_\lambda}$ if $n$ is odd.
Equivalently we may identify
\begin{equation}\label{Eqn: isomorfism between SM on Sigma and SSigma}
	\S\Z|_{\Sigma_t}=
	\begin{cases}
	\S\M_\lambda & \text{if $n$ is even,}\\
	\S\Z|_{\M_\lambda}\oplus \S\Z|_{\M_\lambda} & \text{if $n$ is odd.}
	\end{cases}
\end{equation}
By denoting with $\gamma_\Z$ the Clifford multiplication on $\S\Z$, the family of Clifford multiplications $\gamma_{\lambda}$ satisfies
satisfy
\begin{align}\label{eq:n even}
	&\gamma_{\lambda}(v)\psi=
	\gamma_\Z(e_\lambda) \gamma_\Z(v) \psi
	\qquad &\text{if $n$ is even,}\\ \label{eq:n odd}
	&\gamma_{\lambda}(v)(\psi_++\psi_-)=
	\gamma_\Z(e_\lambda) \gamma_\Z(v) (\psi_+ - \psi_-)
	\qquad &\text{if $n$ is odd,}
\end{align}
where in the second case $\psi=\psi_++\psi_-\in S\Z|_{\M_\lambda}\oplus S\Z|_{\M_\lambda}$ and each component $\psi_\pm$ is identified with an element in $\S^\pm \Z|_{\M_\lambda}$.

\begin{lemma}\label{lem:Kappa}
Let $\Z$ be the Lorentzian manifold given by
$$ \Z=\I \times \M \qquad\qquad g_\Z =  d\lambda^2 + g_\lambda\,,$$
 where $(\M,g_\lambda):=\M_\lambda$ is a family of globally hyperbolic manifolds with $g_\lambda\in{\cal GH}_\M$ for any $\lambda\in\RR$, and denote with $S\M_{\lambda}$ be the spinor bundle over $\M_\lambda$.
For any $p\in\M_{\lambda}$,  the map 
\begin{equation}\label{Eq: isometric bundle isomorphisms}
	 \kappa_{\lambda_1,\lambda_0}\colon \S_{p}\M_{\lambda_0}\to \S_{p}\M_{\lambda_1}\,.
\end{equation}
defined by the   parallel translation on $\Z$
along the curve $\lambda \mapsto (\lambda, p)$
 is a linear isometry and preserves the Clifford multiplication.
\end{lemma}
\begin{proof}
Let $\psi_\alpha\in\Gamma(\S\Z)$ for $\alpha=1,2$ be parallel transported along the curve $\lambda\mapsto(\lambda,p)$, \ie $\nabla_{e_{\lambda}}\psi_\alpha=0$. Since the spin connection preserves the spin product, it follows that
$$
\frac{d}{d\lambda}\fiber{\psi_1(\lambda)}{\psi_2(\lambda)}=\fiber{\nabla_{e_{\lambda}}\psi_1(\lambda)}{\psi_2(\lambda)}+\fiber{\psi_1(\lambda)}{\nabla_{e_{\lambda}}\psi_2(\lambda)}=0\,.
$$
Therefore the spin product on $\S\Z$ is constant along the curve $ \lambda \mapsto (\lambda, p)$ which implies that $\kappa$ is an isometry.
We conclude by showing that the Clifford multiplication is preserved. But this follows from the fact that $e_\lambda$ and the Clifford multiplication $\gamma_\Z$ are parallel along $\lambda\mapsto(\lambda,p)$. Indeed, on account of  the relations~\eqref{eq:n even}-\eqref{eq:n odd}, we get
\begin{align*}
  \kappa_{\lambda_1,\lambda_0}  \Big(\gamma_{{\lambda_0}} (u_{\lambda_0}) \, \psi_{\lambda_0} \Big) &=\kappa_{\lambda_1,\lambda_0}  \Big(\gamma_\Z(e_{\lambda_0})\, \gamma_\Z(u_{\lambda_0}) \, \psi_{\lambda_0} \Big)=\\
&= \Big(\gamma_\Z(\kappa_{\lambda_1,\lambda_0} e_{\lambda_0})\, \gamma_\Z(\kappa_{\lambda_1,\lambda_0} u_{\lambda_0}) \, \kappa_{\lambda_1,\lambda_0} \psi_{\lambda_0} \Big)=\gamma_{\lambda_1}(\kappa_{\lambda_1,\lambda_0} u_{\lambda_0}) \,(\kappa_{\lambda_1,\lambda_0}  \psi_{\lambda_0}) \,,
\end{align*}
so we can conclude our proof.
\end{proof}

Let now $\Dir_\alpha$ be the classical Dirac operator on a global hyperbolic manifold $\M_\alpha$.
With the bundle isomorphism $\kappa_{\lambda_1,\lambda_0}$ we define the intertwining operator 
$$\R_\varrho: \sol_{sc}(\Dir_0)\to \sol_{sc}(\Dir_1)$$ where $\sol_{sc}(\Dir_\alpha)$ is the space of {homogeneous} solutions with spatially compact support, \ie
$$\sol_{sc}(\Dir_\alpha):=\{\psi_\alpha\in \Gamma_{sc}(\S\M_\alpha) \,| \,\psi_\alpha\in \ker\Dir_\alpha   \}\,.$$
Our next task is to show that, assigned $\M_0$ and $\M_1$ there exists a choice of $\varrho$ such that $\R_\varrho$ preserves the  positive definite Hermitian scalar product~\eqref{eq:Herm prod} naturally defined on $\sol_{sc}(\Dir_\alpha)$. We begin by recasting the definition and an important property of the Hermitian scalar product.

\begin{lemma}[\cite{CQF1}, Lemma 3.17]\label{lem:indip Sigma}
Let $\Sigma_\alpha \subset \M_\alpha$ be a smooth
spacelike Cauchy hypersurface with its future-oriented unit normal vector field $\n$ and its induced volume element $\vol_{\Sigma_\alpha}$. Then
\begin{align}\label{eq:Herm prod}
\scalar{\cdot}{\cdot}_\alpha\; \sol_{sc}(\Dir_\alpha)\times \sol_{sc}(\Dir_\alpha) \to \CC \qquad \scalar{\psi_\alpha}{\phi_\alpha}_\alpha=\int_{\Sigma}\fiber{\psi_\alpha}{\gamma_\alpha(\n) \phi_\alpha} \vol_{\Sigma_\alpha}  \,,
\end{align}
yields a positive definite Hermitian scalar product  which does not depend on the choice of $\Sigma_\alpha$.
\end{lemma}

\begin{proposition}\label{prop:conserv scal prod}
Assume the setup of Theorem~\ref{thm:main} and let
 $\varrho \in C^\infty(\M)$ be such that $\vol_{\Sigma_0}=\varrho^2 \vol_{\Sigma_1}$, 
where $\vol_{\Sigma_\alpha}$ is the volume induced by the metric $g_\alpha$ on a Cauchy hypersurface $\Sigma \subset J^-(\Sigma_-)$. 
Then the intertwining operator  $\R_\varrho: \sol_{sc}(\Dir_0)\to \sol_{sc}(\Dir_1)$ preserves the scalar products~\eqref{eq:Herm prod}.
\end{proposition}
\begin{proof}
On account of Proposition~\ref{prop:Dir SHS} and Theorem~\ref{thm:main}, for any $\tilde{\psi}:=\R^+\bkappa_{1,0}\psi_0$ and $\tilde{\phi}:=\R^+\bkappa_{1,0}\phi_0$, by Equation~\eqref{eq:1}, it holds $\R_-\tilde{\psi},\R_-\tilde{\phi}\in\sol_{sc}(\Dir_{1})$. Moreover, by Lemma~\ref{lem:indip Sigma}, the scalar product does not depend on the choice of $\Sigma$, therefore
$$ \scalar{\R_-\tilde{\psi}}{\R_-\tilde{\phi}}_1 = \int_{\Sigma}\fiber{\R_-\tilde{\psi}}{\gamma_1(\n) \R_-\tilde{\psi}} \vol_{\Sigma} = \int_{\Sigma'}\fiber{\R_-\tilde{\psi}}{\gamma_1(\n) \R_-\tilde{\psi}} \vol_{\Sigma'}\,.$$
By choosing $\Sigma'\subset J^+(\Sigma_+)$ we have $\chi=1$, therefore the Hermitian scalar product reads as
$$ \scalar{\R_-\tilde{\psi}}{\R_-\tilde{\phi}}_1  = \int_{\Sigma'}\fiber{\tilde{\psi}}{\gamma_1(\n) \tilde{\psi}} \vol_{\Sigma'}\,.$$
On account of Equation~\eqref{eq:2}, we get
 $\tilde{\psi},\tilde{\phi}\in\sol_{sc}(\Dir_{\chi,1}^\varrho)$. This implies that latter Hermitian form can be read as the scalar product $\scalar{\cdot}{\cdot}_{\chi,1}$ on the space of solution $\Dir_{\chi,1}^\varrho$,
\ie
$$\scalar{\R_-\tilde{\psi}}{\R_-\tilde{\phi}}_1  = \int_{\Sigma'}\fiber{\tilde{\psi}}{\gamma_1(\n) \tilde{\psi}} \vol_{\Sigma'}=\scalar{\tilde{\psi}}{\tilde{\phi}}_{\chi,1}\,.$$ 
  This because on $\Sigma'$, we have $\chi=1$ which implies that $\Dir_{\chi=1,1}^\varrho=\Dir_1$ on small tubular neighborhood of $\Sigma'$ contained in the future of $\Sigma_+$.
   By Lemma~\ref{lem:indip Sigma}, we thus obtain
 $$ \scalar{\tilde{\psi}}{\tilde{\phi}}_{\chi,1}= \int_{\Sigma'}\fiber{\tilde{\psi}}{\gamma_{1,1}(\n) \tilde{\psi}} \vol_{\Sigma'}= \int_{\Sigma''}\fiber{\tilde{\psi}}{\gamma_{\chi,1}(\n) \tilde{\psi}} \vol_{\Sigma''}\,.$$
 By choosing $\Sigma''=\Sigma$, which lies in the past of $\Sigma_-$, we obtain
 $$ \scalar{\tilde{\psi}}{\tilde{\phi}}_{\chi,1}= \int_{\Sigma}\fiber{\bkappa_{1,0}{\psi_0}}{\gamma_{0,1}(\n) \bkappa_{1,0}{\psi_0}} \vol_{\Sigma}=\scalar{\bkappa_{1,0}{\psi_0}}{\bkappa_{1,0}{\psi_0}}_{0,1}\,,$$
 where the latter forms is defined on the solution space of the operator $\Dir_{0,1}$.
But, on account of Lemma~\ref{lem:Kappa}, $\kappa_{1,0}$ is an isometry of spinor bundles which preserves the Clifford multplication. Therefore it follows
$$\scalar{\bkappa_{1,0} \psi_0}{\bkappa_{1,0} \phi_0}_{0,1} = \int_{\Sigma}\fiber{\kappa_{1,0} \psi_0}{\gamma_{0,1}(\n) \kappa_{1,0} \phi_0} \varrho^2\vol_{\Sigma_1} = \scalar{\psi_0}{ \phi_0}_0\,. $$
Plugging all together we can conclude.
\end{proof}

\subsubsection{Symplectic structures for geometric wave operator}\label{sec:Scal-sol-herm}

Let $\M_\alpha$ be globally hyperbolic manifolds with the same Cauchy temporal function and let $\P_\alpha$ be a normally hyperbolic operator acting on section of a vector bundle $\V_\alpha$ with finite rank and with a positive definite Hermitian form $\fiber{\cdot}{\cdot}_\alpha$. On account of Section~\ref{sec:GWO}, $\P_\alpha$ can be reduced to a symmetric hyperbolic operator $\oS_\alpha$ acting on sections of the vector bundle $\E_\alpha=\V_\alpha\oplus(\T^*\Sigma\otimes \V_\alpha)\oplus \V_\alpha$, and the map $\iota: \sol(\P_\alpha)\to\sol(\oS_\alpha) $ is injective. We denote with $\sol_{\P_\alpha}(\oS_\alpha)$ the subspace of $\sol(\oS_\alpha)$ where $\iota$ is bijective, \ie $\sol_{\P_\alpha}(\S_\alpha):=\iota(\sol(\P_\alpha))$. 
\begin{proposition}\label{prop:overR}
Let $\alpha\in\{0,1\}$ and consider $\P_\alpha$ and $\S_\alpha$ as above. If there exists a linear isometry $\overline\kappa_{1,0}:\V_0\to \V_1$, then there exists an intertwining operator $\overline{\R}_\varrho:=\sol(\P_0)\to\sol(\P_1)$ which implements a non-canonical isomorphism. 
\end{proposition}
\begin{proof}
Since the tangent bundle does not depend on the underline metric, then we can define $\kappa_{1,0}:\E_0\to\E_1$ as
$$\begin{pmatrix}
f_0 \\ v \otimes f'_0 \\ f''_0 
\end{pmatrix} \mapsto \begin{pmatrix}
\overline{\kappa}_{1,0}f_0 \\ v \otimes \overline{\kappa}_{1,0}f'_0 \\ 	\overline{\kappa}_{1,0}f''_0 
\end{pmatrix} \,.
$$
Since $\overline{\kappa}_{1,0}$ is an isometry, then also $\kappa_{1,0}$ enjoys the same property.
Hence the operator $\overline{\R}_\varrho$ which makes the following diagram commutative implements the desired non-canonical isomorphism
\begin{flalign*}
\xymatrix{
\sol(\P_0) \ar[d]_-{\iota} \ar[rr]^-{\overline{\R}_\varrho} && \sol(\P_1) \ar[d]_-{\iota}   \\
\sol_{\P_0}(\oS_0)  \ar[rr]^-{\R_\varrho} &&  \sol_{\P_1}(\oS_1) 
}
\end{flalign*}
where $\R_\varrho$ is given as in Theorem~\ref{thm:main}.
\end{proof}

\begin{corollary}\label{cor:R wave}
Assume the setup of Theorem~\ref{thm:main} and Proposition~\ref{prop:overR}.
The operator $\overline{\R}_\varrho$ reads as:
$$\overline\R_\varrho:=\overline\R_-\circ \overline\R_+ \circ  \overline{\kappa}^\varrho_{1,0}$$
where 
$$\overline\R_-:=\Id - \G_{1}^- (1-\chi) (\P_1 - \P_{0,1}^\varrho) \quad \text{and}\quad  \overline\R_+:=\Id - \G_{\chi}^+ \chi (\P_1 -\P_{0,1}^\varrho) $$
where $\G^+_\chi$ is the advanced Green operator for $\P_{\chi,1}^\varrho$ and $\G^-_1$ is the retared Green operator for $\P_1$.
\end{corollary}
\begin{proof}
We begin by noticing that the principal symbols of $\P_{\chi,1}^\varrho:= (1- \chi) \,\P^\varrho_{0,1} + \chi\P_1$ it satifies
\begin{align*}
\sigma_{\P_{\chi,1}^\varrho}(df) u_1 &= \frac{1}{2}\P_{\chi,1}^\varrho(f^2) u_1 + \frac{1}{2} f^2 \P_{\chi,1}^\varrho(u_1) = (1-\chi)\sigma_{\P_{0,1}^\varrho}(df)u_1 +\chi\sigma_{\P_{1}}(df)u_1 \\
&=- \Big( (1-\chi) g_0(df,df) \, + \chi g_1(df,df) \Big) u_1\,,
\end{align*}
where we used that $ \sigma_{\P^\varrho_{0,1}} = \overline{\kappa}^\varrho_{1,0}\sigma_{\P_{0}} \overline\kappa^\varrho_{0,1}$, see e.g. Equation~\ref{eq:transf princ symb}. By defining $h:=(1-\chi)g_0+\chi g_1$, computations analogous to the one in Section~\ref{sec:GWO}, shows that $\P_{\chi,1}^\varrho$ can be reduced to a symmetric hyperbolic system. This is enough to guaranteed the existence of solutions and hence Green operators for $\P_{\chi,1}^\varrho$ as well. Using analogous computations to the ones performed in the proof of Theorem~\ref{thm:main} we can already conclude.
\end{proof}

As for Section~\ref{sec:Dir-Sol-Herm}, let $\sol_{sc}(\P_\alpha)$ be the space of homogeneous solutions with spacially compact support, \ie
$$\sol_{sc}(\P_\alpha):=\{u_\alpha\in \Gamma_{sc}(\V_\alpha) \,| \, u_\alpha\in \ker\P_\alpha \}\,.$$
Our next task is to show that, assigned $\M_0$ and $\M_1$ there exists a choice of $\varrho$ such that $\overline\R_\varrho$ preserves the symplectic form naturally defined on $\sol_{sc}(\P_\alpha)$. As for the scalar product~\eqref{eq:Herm prod}, also the symplectic form does not depends on the choice of the Cauchy hypersurface.

\begin{lemma}[\cite{CQF1}, Lemma 3.17]\label{lem:indip Sigma 2}
Let $\Sigma_\alpha \subset \M_\alpha$ be a smooth
spacelike Cauchy hypersurface with its future-oriented unit normal vector field $\n$ and its induced volume element $\vol_{\Sigma_\alpha}$. Then
\begin{equation}\label{eq:sympl form}
\begin{aligned}
 \bracket{\cdot}{\cdot}_\alpha &\;:  \sol_{sc}(\P_\alpha)\times \sol_{sc}(\P_\alpha) \to \CC \\
  \bracket{v_\alpha}{v_\alpha}_\alpha=&\int_{\Sigma}\fiber{v_\alpha}{\nabla_\n u_\alpha}-\fiber{u_\alpha}{\nabla_\n v_\alpha} \vol_{\Sigma_\alpha}  \,,
\end{aligned}
\end{equation}
yields  a symplectic form  which does not depend on the choice of $\Sigma_\alpha$.
\end{lemma}
Using the same arguments as in the proof of Proposition~\ref{prop:conserv scal prod} we can conclude an analog conservation of the symplectic form.
\begin{proposition}\label{prop:conserv sympl form}
Assume the setup of Corollary~\ref{cor:R wave} and let
 $\varrho \in C^\infty(\M)$ be such that $\vol_{\Sigma_0}=\varrho^2 \vol_{\Sigma_1}$, 
where $\vol_{\Sigma_\alpha}$ is the volume induced by the metric $g_\alpha$ on a Cauchy hypersurface $\Sigma \subset J^-(\Sigma_-)$. 
Then the intertwining operator  $\overline\R_\varrho$  preserves the symplectic form~\eqref{eq:sympl form}.
\end{proposition}

\section{Applications}\label{sec:applications}

We conclude this paper with an application inspired by~\cite{Moller,simo3}. In \textit{loc. cit.}, the quantization of a free field theory is interpreted as a two-step procedure:
\begin{itemize}
\item[1.] The first consists of the assignment to a physical system of a $*$-algebra of observables which encodes structural properties such as causality, dynamics and the canonical commutation/anti-commutation relations.
\item[2.] The second step calls for the identification of an algebraic state, which is a positive, linear and normalized functional on the algebra of observables.
\end{itemize}
\begin{remark}
This quantization scheme goes under the name of \emph{algebraic quantum field theory} and it is especially well-suited for formulating quantum theories also on manifold -- see e.g.~\cite{gerard, aqft2} for textbook, to~\cite{CQF1,CQF2,FK,BD, ThomasAlex} for recent reviews, \cite{aqftCAT,aqftCAT2,aqftCAT3,aqftCAT4} for homotopical approaches and~\cite{Nic,Buc,cap,simo4,DHP,DMP3, CSth,cospi,simo1,simo2,simo3,SDH12} for some applications. 
\end{remark}

 As explained in details by Benini, Dappiaggi and Schenkel in~\cite{aqftAff}, the space of observables for the inhomogeneous solution coincides with the one of the homogeneous solution plus an extra observable assigned to the particular solution $\G^+ \ff$, where $\ff$ is the source term and $\G^+$ is the advanced Green operator for the inhomogeneous Cauchy problem. Therefore, in what follows we shall only consider the quantization of  classical fields which satisfies an homogeneous Cauchy problem. 

\subsection{Algebra of Dirac fields}\label{sec:alg Dirac}

As in Section~\ref{sec:Dir-Sol-Herm}, let $\sol_{sc}(\Dir)$ be the space of \emph{homogeneous} solutions with spatially compact support of the Dirac operator endowed with the positive definite Hermitian scalar product~\eqref{eq:Herm prod}.

\begin{definition}\label{def:alg Dirac}
We call \emph{algebra of Dirac fields} the unital, complex $*$-algebra $\mathfrak{A}$ freely generated by  the abstract elements $\Xi(\psi)$, with $\psi\in\sol_{sc}(\Dir)$, together with the following relations for all $\psi,\phi\in\sol_{sc}(\Dir)$ and $\alpha,\beta\in\CC$:
\begin{itemize}
\item[(i)] Linearity: $\Xi(\alpha \psi + \beta \phi) =\alpha \Xi(\psi) + \beta\Xi(\phi)$
\item[(ii)] Hermiticity: $\Xi(\psi)^*=\Xi(\Upsilon\psi)$
\item[(iii)] Canonical anti-commutation relations (CARs):
$$ \Xi(\psi) \cdot\Xi(\phi) + \Xi(\phi)\cdot \Xi(\psi) =0 \qquad \text{and} \qquad  \Xi(\psi) \cdot\Xi(\phi)^* + \Xi(\phi)^*\cdot \Xi(\psi) = \scalar{\psi}{\phi} \,\Id_\fA \,$$
where $\Upsilon$ is the adjunction map~\ref{eq:adj map} and $\scalar{\cdot}{\cdot}$ is the Hermitian scalar product~\eqref{eq:Herm prod}.
\end{itemize}
\end{definition}

A more concrete construction can be obtained as follow. Denote with 
$$\sol_\oplus:=\sol_{sc}(\Dir)\oplus \sol_{sc}(\Dir^*)$$
and consider the tensor $\CC$-algebra $  \mathfrak{T}:= \Big( \bigoplus_{n\in\mathbb{N}}  (\sol_\oplus)^{\otimes_{n}}, \; \mathsmaller{\bullet} \; \Big) \,. $
Notice that, on account of Remark~\ref{rmk:soldual}, $\sol_{sc}(\Dir^*)=\Upsilon\sol_{sc}(\Dir)$. The generators of $\mathfrak{T}$ are given by
$$ \Id_{\mathfrak{T}}= \left\lbrace 1,  0 \dots \right\rbrace \qquad \Xi(\psi):= \left\lbrace 0,  \begin{pmatrix} \psi \\ 0  \end{pmatrix} , 0 \dots \right\rbrace \qquad 
\Xi(\phi)^*:= \left\lbrace 0,  \begin{pmatrix} 0 \\ \Upsilon\phi   \end{pmatrix} , 0 \dots \right\rbrace $$
for any $\psi,\phi\in\sol_{sc}(\Dir)$ and the involution $*: \mathfrak{T}\to\mathfrak{T}$
 is implemented by means of the antilinear isomorphism $\Upsilon$
$$  \left\lbrace 0,\dots,  \begin{pmatrix} \psi_1 \\ \Upsilon\phi_1  \end{pmatrix}\otimes\cdots\otimes \begin{pmatrix} \psi_k \\ \Upsilon\phi_k  \end{pmatrix} , 0 \dots \right\rbrace^* = \left\lbrace 0,\dots,  \begin{pmatrix} \phi_k \\ \Upsilon\psi_k  \end{pmatrix}\otimes\cdots\otimes \begin{pmatrix} \phi_1 \\ \Upsilon\psi_1  \end{pmatrix} , 0 \dots \right\rbrace $$
for every $\psi_1,\phi_1,\dots\psi_k,\phi_k\in\sol_{sc}(\Dir)$. As always, $*$ 
is extended
to all elements of $\mathfrak{T}$ by antilinearity, thus turning $\mathfrak{T}$ into a unital complex $*$-algebra. The canonical relations are implemented taking the quotient of $\mathfrak{T}$  by the $*$-ideal $\mathfrak{I}$ generated by
\begin{equation}\label{eq:CAR}
 \Xi(\psi)\, \mathsmaller{\bullet}\, \Xi(\phi) + \Xi(\phi)\,\mathsmaller{\bullet} \, \Xi(\psi)  \qquad \text{and} \qquad  \Xi(\psi)\, \mathsmaller{\bullet}\, \Xi(\phi)^* + \Xi(\phi)^*\mathsmaller{\bullet} \, \Xi(\psi) - \scalar{\psi}{\phi} \,\Id_\fA \,
 \end{equation}
for every $\psi,\phi\in\sol_{sc}(\Dir)$.
If follows that $\fA=\mathfrak{T}/\mathfrak{I}$ is a  realization of the algebra of Dirac fields. \medskip

Keeping in mind this concrete realization, we can prove the following isomorphism.

\begin{theorem}\label{thm:alg iso}
Assume the setup Proposition~\ref{prop:conserv scal prod} and let $\fA_\alpha$ be the algebra of Dirac fields on $\M_\alpha$. Then there exists a $*$-isomorphism $\fR_{1,0}:\fA_0\to\fA_1$.
\end{theorem}
\begin{proof}
Theorem~\ref{thm:main} and Proposition~\ref{prop:R*} establish via $\R_\varrho$ and $\R_\varrho^*$ an isomorphism between $\sol(\Dir_\alpha)$ and $\sol(\Dir_\alpha^*)$ respectively. As a by-product, $\R_\varrho\oplus \R_\varrho^*$ extends first of all to an isomorphism between the tensor algebras $\mathfrak{T}(\sol^\oplus_\alpha)$ by linearity. Finally, on account of Proposition~\ref{prop:conserv scal prod} $\R_\varrho^*= \Upsilon_1\R_\varrho \Upsilon_0$ and $\R_\varrho$ preserve the Hermitian scalar product, which implies that the ideals $\mathfrak{I}_\alpha$ are $*$-isomorphic.
\end{proof}

Before concluding this subsection, we want to make the following remark:
\begin{remark}
The algebra of Dirac fields cannot be considered an algebra of observables, since observables are required to commute at spacelike separations and $\fA$ does not fulfil such requirement. However, the subalgebra $\fA_{\text{obs}} \subset\fA$  composed by even elements, \ie $\Xi(\psi)=-\Xi(\psi)$, which are invariant under the action of $\Spin_0(1,n)$ (extended to $\fA$) is a good candidate as algebra of observables. For further details we refer to~\cite{DHP,dimock}.
\end{remark}

\subsection{Algebras of real scalar fields}\label{sec:alg scalar}
Consider a normally hyperbolic operator $\P$ acting on a real line bundle $\V:=\RR\times\M$. As in Section~\ref{sec:Scal-sol-herm}, let $\sol_{sc}(\P)$ be the space of \emph{homogeneous} solutions with spatially compact support endowed with the symplectic form~\eqref{eq:Herm prod}.

\begin{definition}\label{def:alg scalar}
We call \emph{algebra of real scalar fields} the unital, complex $*$-algebra $\mathfrak{A}$ freely generated by  the abstract elements $\Phi(u)$, with $u\in\sol_{sc}(\P)$, together with the following relations for all $u,v\in\sol_{sc}(\P)$ and $\alpha,\beta\in\CC$:
\begin{itemize}
\item[(i)] Linearity: $\Phi(\alpha u + \beta v) =\alpha \Phi(u) + \beta\Phi(v)$
\item[(ii)] Hermiticity: $\Phi(u)^*=\Phi(u)$
\item[(iii)] Canonical commutation relations (CCRs):
$$ \Phi(v) \cdot\Phi(u) - \Phi(u)\cdot \Phi(v)  = \bracket{v}{u} \,\Id_\fA \,$$
where $\bracket{\cdot}{\cdot}$ is the symplectic form~\eqref{eq:Herm prod}.
\end{itemize}
\end{definition}

A more concrete construction can be obtained by mimicking the one for the Dirac fields. First, consider  the tensor $\CC$-algebra $  \mathfrak{T}:= \Big( \bigoplus_{n\in\mathbb{N}}  \sol_{sc}(\P)^{\otimes_{n}}, \; \mathsmaller{\bullet} \; \Big) \,, $
where the generators of $\mathfrak{T}$ are given by
$$ \Id_{\mathfrak{T}}= \left\lbrace 1,  0 \dots \right\rbrace \qquad \Phi(u):= \left\lbrace 0,  u, 0 \dots \right\rbrace $$
for any $u\in\sol_{sc}(\Dir)$ and the involution $*:\mathfrak{T}\to\mathfrak{T}$
 is implemented by means of the antilinear isomorphism $\Upsilon$
$$  \left\lbrace 0,\dots, u_1\otimes\cdots\otimes u_k , 0 \dots \right\rbrace^* = \left\lbrace 0,\dots,  u_k\otimes\cdots\otimes u_1 , 0 \dots \right\rbrace $$
for every $\psi_1,\phi_1,\dots\psi_k,\phi_k\in\sol_{sc}(\Dir)$. As always, $*$ 
is extended to all elements of $\mathfrak{T}$ by antilinearity, thus turning $\mathfrak{T}$ into a unital complex $*$-algebra. The canonical commutation relations are implemented taking the quotient of $\mathfrak{T}$  by the $*$-ideal $\mathfrak{I}$ generated by
\begin{equation}\label{eq:CCR}
 \Phi(u)\, \mathsmaller{\bullet}\, \Phi(v) - \Phi(v)\,\mathsmaller{\bullet} \, \Phi(u)   - \bracket{u}{v} \,\Id_\fA \,
\end{equation}
for every $u,v\in\sol_{sc}(\P)$.
If follows that $\fA=\mathfrak{T}/\mathfrak{I}$ is a  realization of the algebra of real scalar fields. \medskip

Keeping in mind this concrete realization, we can prove the following isomorphism.

\begin{theorem}\label{thm:alg iso2}
Assume the setup Proposition~\ref{prop:conserv sympl form} and let $\fA_\alpha$ be the algebra of real scalar fields on $\M_\alpha$. Then there exists a $*$-isomorphism $\fR_{1,0}:\fA_0\to\fA_1$.
\end{theorem}
\begin{proof}
Corollary~\ref{cor:R wave} establish via $\overline\R_\varrho$ an isomorphism between $\sol(\P_\alpha)$. As a by-product,
such result extends first of all to an isomorphism between the tensor algebras $\mathfrak{T}(\sol_{sc}(\P_\alpha))$ by linearity. Finally, by Proposition~\ref{prop:conserv sympl form} $\overline{\R}_\varrho$ preserves the Hermitian scalar product, which implies that the ideals $\mathfrak{I}_\alpha$ are $*$-isomorphic.
\end{proof}

\subsection{Hadamard states}

We conclude this section by studying (algebraic) states and their interplay with the intertwining operator $\mathfrak{R}$.
\begin{definition}
Given a complex $*$-algebra $\mathfrak{A}$ we call  \emph{(algebraic) state} any linear functional from $\mathfrak A $ into $\CC$ that is positive, \ie $\omega(\aa^*\aa)\geq 0$ for any $\aa\in\fA $, and normalized, \ie $\omega(\Id_\fA)=1$. 
\end{definition}
Let $\V$ be a vector bundle and consider a Green hyperbolic operator $\Q:\V\to\V$. 
Let $\sol_{sc}(\Q)$ be the space of homogeneous solutions with spacially compact support. 
\begin{definition}
Following Section~\ref{sec:alg Dirac} and Section~\ref{sec:alg scalar} we call \emph{algebra of fields} 
$$\fA=\frac{\bigoplus_{n\in\mathbb{N}} \sol_{sc}(\Q)^{\otimes_{n}}}{\fI }$$
where $\fI$ is a suitable $*$-ideal which encodes CCR- or CAR-relations. 
\end{definition}
\begin{remark}
Notice that if we set $\Q=\P:C^\infty(\M)\to C^\infty(\M) $ and we consider the ideal generated by the CCR relations~\eqref{eq:CCR}, we thus obtain the algebra of real scalar fields. While if we set $\Q=\Dir\oplus\Dir^*:\Gamma(\S\M)\oplus\Gamma(\S^*\M)\to\Gamma(\S\M)\oplus\Gamma(\S^*\M)$ and we consider the ideal generated by the CAR relations~\eqref{eq:CAR}, we thus obtain the algebra of Dirac fields.
\end{remark}
Due to the natural grading on the algebra of fields $\fA$, it suffices to $\omega$ on the monomials. This gives rise to the n-points distributions $\omega^{(n)}\in \big( \Gamma_c(\V)\big)'$ by means of the relations
\begin{equation}\label{eq: omega_n}
\omega^{(n)}\left(f_1,\dots,f_n\right)=\omega\left(f_1 \otimes \dots \otimes f_n \right) :\frac{ (\sol_{sc}(\Q))^{\otimes_{n}}}{\mathfrak{I}} \to \CC \,.
\end{equation}
where $u_j=\G f_j$, with $j=1,\dots,n$, and $\G$ is the causal Green propagators for $\Q$.
 This leads us to the following definition.

\begin{definition}\label{def:quasifree}
A state $\omega$ on the algebra of fields is \emph{quasifree} if its n-point functions 
$\omega^{(n)}$   vanish for odd n, while for even n, they are defined  as
\[ \omega^{(n)}\left(f_1,\dots,f_n\right)=\sum\limits_{\sigma \in S'_n} (-1)^{\text{\rm{sign}}(\sigma)} \prod\limits_{i=1}^{n/2}
\;\omega^{(2)} \left(
f_{\sigma(2i-1)} ,
f_{\sigma(2i)} \right) \:, \]
where~$S'_n$ denotes the set of ordered permutations of $n$ elements.
\end{definition}

\begin{remark}\label{rmk:Had cond}
It is widely accepted that, among all possible (quasifree) states, the physical ones are required to satisfy the so-called \emph{the Hadamard condition}.  
The reasons for this choice are manifold: For example, it implies the finiteness of the quantum fluctuations of the expectation value of every observable and it allows us to construct Wick polynomials following a covariant scheme, see \cite{HW} or \cite{IgorValter} for recent reviews.
This requirement is conveniently translated in the language of microlocal analysis, in particular into a microlocal characterization of the two-points distribution of the state.  Since a full characterization is out of the scope of the paper, for further details we refer to~\cite{cap,gerard1,gerard2,gerard3}  for scalar fields and to~\cite{simo1,DHP} for Dirac fields -- see also~\cite{cap2, gerardYM,simogravity,Micau3} for gauge theory.
\end{remark}

With the next theorem, we show that the pull-back of a quasifree state along the isomorphism $\mathfrak{R}_{1,0}:\fA_0 \to \fA_1$ induced by the intertwining operator $\R_\varrho$ for $\Q$  (see e.g. Theorem~\ref{thm:alg iso} or Theorem~\ref{thm:alg iso2}) preserves singularity structure of the two-point distribution $\omega^{(2)}$, \ie  it preserves the wavefront set.  

\begin{theorem}\label{thm:main appl}
Assume the Setup~\ref{setup} and denote with $\fA_\alpha$ the algebra of fields on $\M_\alpha$. Finally let $\omega_\alpha:\fA_\alpha\to\CC$ be quasifree states satisfying
$$\omega_0=\omega_1 \circ \mathfrak{R}_{1,0}\,: \fA_0 \to \CC$$
with $\mathfrak{R}_{1,0}$ is the isomorphism induced by $\R_\varrho$. Then the bi-distributions of the associated two-point distributions $\omega^{(2)}_\alpha$  have the same singularity structure. 
 \end{theorem}

\begin{proof}
Since $\fR_{1,0}$ is $*$-isomorphism, $\omega_0$ inherits the property of being a quasifree state from $\omega_1$. 
In particular two-point function $\omega_0^{(2)}$ satisfies
\begin{align*}
\omega_0^{(2)} \left( f_0,g_0\right):=\omega_0 \left( 
u_0\otimes v_0\right)&=\omega_1^{(2)} \left( \R_\varrho u_0\otimes \R_\varrho v_0\right)\,.
\end{align*}
Consider now the restriction of $\omega^{(2)}_1$ to a neighborhood $\N_1$ of a Cauchy hypersurface ${\Sigma_1}\subset J^+(\Sigma_+)$. On account of Theorem~\ref{thm:main}, $\chi=1$ and we have ${\R_\varrho}|_{\N_1}= \R_+\kappa_{1,0}^\varrho$\,.
In particular, $\omega_{0}$ reads as
\begin{align*}
\omega_0^{(2)} \left.\left( f_0,g_0\right)\right|_{\N_1}&=\omega_1 \left.\left( \R_+\kappa_{1,0}^\varrho u_0 \otimes
\R_+\kappa_{1,0}^\varrho v_0 \right)\right|_{\N_1}\,.
\end{align*}
The latter two-points distribution can also be read as the restriction to $\N$ of the two-point distribution $\omega_\chi$ of a state on the algebra of $\Q_\chi$-fields (\cf Proof of Theorem~\ref{thm:main}), namely
$$\omega_\chi:\fA_\chi\to \CC \qquad \omega_\chi \Big|_{\N_1}= \Big(\omega_1 \circ \fR_{1,0} \Big) \Big|_{\N_1}\,. $$
 On account of the Theorem on the propagation of singularites, see e.g.\cite{propSing}, $\omega_\chi^{(2)}$ has the same singularities in the whole manifold. By restricting $\omega_\chi$ to a neighborhood $\N_2$ of a Cauchy hypersurface ${\Sigma_2}\subset J^-(\Sigma_-)$ 
we thus obtain
$$ \omega^{(2)}_\chi \Big|_{\N_2}= \omega_1^{(2)} \left.\left( 
\kappa_{1,0}^\varrho f_0 , \kappa_{1,0}^\varrho g_0 \right)\right|_{\N_2}\,. $$
Since the composition with a smooth map does not change the wavefront set, $\omega_0^{(2)}$ has the same singularity structure of $\omega_1^{(2)}$. 
\end{proof}

\begin{remark}
The main drawback of the definition of the intertwining operator $\mathfrak{R}$, used in Theorem~\ref{thm:main appl}, is the lack of any control on the action of the group of $*$-automorphism induced by the isometry group of $\M$ on $\omega_2$. Let us remark, that the study of invariant states is a well-established research topic (\cf \cite{NCTori,NCspinoff}). Indeed, the type of factor can be inferred by analyzing which and how many states are invariant. From a more physical perspective instead, invariant states can represent equilibrium states in statistical mechanics e.g. KMS-states or ground states.
\end{remark}
The previous remark leads us to the following open question: Under which conditions it is possible to perform an adiabatic limit, namely when $ \lim_{\chi\to 1} \omega_1$ is well-defined? \medskip

A priori we expect that there is no positive answer in all possible scenarios, since
it is known that certain free-field theories, e.g., the massless and minimally coupled (scalar or Dirac) field on four-dimensional de Sitter spacetime, do not possess a ground state, even though their massive counterpart does. A partial answer is given in~\cite{Moller,Moller2} for the case of scalar field theory on globally hyperbolic manifolds. In those papers, it is investigated how to relate normally hyperbolic operators which differ from a smooth potential, e.g. massive and massless wave operators.
\medskip

We conclude this paper with the following corollary, which is a new proof of the existence of Hadamard states on every globally hyperbolic manifold.
\begin{corollary}\label{cor:existence Hadamard}
Let $(\M,g)$ be a globally hyperbolic manifold  and denote with $\mathfrak A$ the algebra of fields. Then there exists a state $\omega:\mathfrak{A}\to\CC$ which satisfies the Hadamard condition.
\end{corollary}
\begin{proof} 
Let $\Sigma$ be a Cauchy hypersurface for $\M$ and denote with $\mathcal{O}$ a globally hyperbolic open neighborhood of $\Sigma$, namely $\cal O$ is an open neighborhood of $\Sigma$ in $\M$ containing all causal curves for $\M$ whose endpoints lie in $\cal O$. Notice that $\cal O$ is a globally hyperbolic submanifold of $\M$ and it posses an algebra of fields which we denote by $\fA_{\cal O}$.
Consider now $g|_{\cal O}$, the restriction of $g$ to $\cal O$ and denote with $g_u$ a static metric such that $g_u\leq g|_{\cal O}$ and $g_u,g\in\cal{GH}_{\cal O}$. On a static globally hyperbolic manifold $({\cal O},g_u)$ one may construct a ground state $\omega_H$ for the algebra of fields $\fA^u_{\cal O}$ which can be shown to be Hadamard (\cf Remark~\ref{rmk:Had cond}). 
By denoting with $\fR:\fA_{\cal O}\to\fA^u_{\cal O}$ the intertwining operator between the algebra of fields on $({\cal O},g|_{\cal O})$ and $({\cal O},g_u)$ respectively (\cf Theorem~\ref{thm:alg iso} and Theorem~\ref{thm:alg iso2}), the state defined by
$$\omega=\omega_H \circ \mathfrak{R}\,: \fA_{\cal O} \to \CC$$
satisfies the Hadamard condition on account of Theorem~\ref{thm:main appl}.  On account of the time-slice axiom (see e.g.~\cite[Section 3]{BD}), $\fA_{\cal O}$   and $\fA$ are $*$-isomorphic, so we can pull-back $\omega$ along the $*$-isomorphism and therefore we can conclude.
\end{proof}

\end{document}